\documentclass[11pt]{article}
\usepackage{graphicx} 
\usepackage[english]{babel}
\usepackage{mathtools}
\usepackage{amsfonts}
\usepackage{amsmath}
\usepackage[margin=1in]{geometry}
\usepackage{amsthm}
\usepackage{hyperref}
\hypersetup{colorlinks=true,linkcolor=black,citecolor=[rgb]{0.8,0.1,0.3}}
\usepackage{amssymb}
\usepackage{dsfont}
\usepackage{xspace}
\usepackage{gensymb}

\newcommand{\norm}[1]{\lvert \lvert #1 \rvert \rvert}
\newcommand{\dotproduct}[2]{\vec{#1} \cdot \vec{#2}}
\newcommand{\vvector}[3]{\left( \underbrace{#1,\cdots}_{#2},#3,\cdots,#3\right)}

\usepackage{tikz}
\usepackage{wrapfig}
\usepackage{subcaption}
\usepackage{comment}
\usepackage{empheq}

\graphicspath{{images/}}

\newtheorem{theorem}{Theorem}[section]
\newtheorem{claim}{Claim}[section]
\newtheorem{conjecture}[theorem]{Conjecture}
\newtheorem{lemma}[theorem]{Lemma}
\newtheorem{proposition}[theorem]{Proposition}
\newtheorem{corollary}[theorem]{Corollary}
\newtheorem{remark}[theorem]{Remark}
\newtheorem{observation}[theorem]{Observation}
\newtheorem{open}{Open Question} 

\theoremstyle{definition}
\newtheorem{definition}{Definition}[section]

\DeclareMathOperator{\poly}{poly}

\newcommand{\apx}{$\mathsf{APX}$\xspace}
\newcommand{\ptas}{$\mathsf{PTAS}$\xspace}
\newcommand{\np}{$\mathsf{NP}$\xspace}
\DeclareMathOperator{\R}{\mathbb{R}}
\usepackage{todonotes}

\usepackage{url}
\usepackage{setspace}
\title{\textbf{On Steiner Trees of the Regular Simplex}}
\author{Henry Fleischmann\\ University of Cambridge\and Guillermo A. Gamboa Q. \\ Charles University\and Karthik C.\ S.\\ Rutgers University\and  Josef Matějka\\ Charles University\and Jakub Petr\\ Charles University}

\date{}
\begin{document}

\maketitle

\begin{abstract}
In the Euclidean Steiner Tree problem, we are given as input a set of points (called \emph{terminals}) in the $\ell_2$-metric space and the goal is to find the minimum-cost tree connecting them. Additional points (called \emph{Steiner points}) from the space can be introduced as nodes in the solution. \vspace{0.1cm}

The seminal works of Arora [JACM'98] and Mitchell [SICOMP'99] provide a Polynomial Time Approximation Scheme (PTAS) for solving the Euclidean Steiner Tree problem in fixed dimensions. However, the problem remains poorly understood in higher dimensions (such as when the dimension is logarithmic in the number of terminals) and ruling out a PTAS for the problem in high dimensions is a notoriously long standing open problem (for example, see Trevisan [SICOMP'00]).
Moreover, the explicit construction of optimal Steiner trees remains unknown for almost all well-studied high-dimensional point configurations. Furthermore, a vast majority the state-of-the-art structural results on (high-dimensional) Euclidean Steiner trees were established in the 1960s, with no noteworthy update in over half a century.\vspace{0.1cm}

In this paper, we revisit high-dimensional Euclidean Steiner trees, proving new structural results. We also establish a link between the computational hardness of the Euclidean Steiner Tree problem and understanding the optimal Steiner trees of regular simplices (and simplicial complexes), proposing several conjectures and showing that some of them suffice to resolve the status of the inapproximability of the Euclidean Steiner Tree problem. Motivated by this connection, we investigate optimal Steiner trees of regular simplices, proving new structural properties of their optimal Steiner trees, revisiting an old conjecture of Smith [Algorithmica'92] about their optimal topology, and providing the first explicit, general construction of candidate optimal Steiner trees for that topology.
\end{abstract}

 \thispagestyle{empty}
 
\clearpage
 \setcounter{page}{1}

\section{Introduction} \label{sec: intro}
Given a set of points in space (called \textit{terminals}), a \textit{Steiner tree} of those points is a tree connecting those points. In addition to the terminals, the tree may contain additional points from the ambient space (called \textit{Steiner points}). Finding the minimum cost Steiner tree is one of the most fundamental problems in Computer Science, Operations Research, and Combinatorial Optimization \cite{ljubic2021solving}. For example, Steiner trees arise naturally in network design, the design of integrated circuits, location problems, machine learning, computer vision, systems biology, and bioinformatics \cite{cheng2013steiner,cho2001steiner,hwang1992steiner,lengauer2012combinatorial,resende2008handbook,noormohammadpour2017dccast, backes2012integer,ideker2002discovering,russakovsky2010steiner, tuncbag2016network}.

In this work, we focus on the Euclidean Steiner tree problem, perhaps the most fabled setting of the problem, where the terminals lie in the Euclidean metric space. It was first studied in full generality at least as far back as 1811 and has been discussed in letters of Gauss. For three points, the Fermat-Toricelli problem, optimal Steiner trees were characterized completely as early as the 1600s by Toricelli. The interested reader may see \cite{brazil2014history} for more details on the history of the Euclidean Steiner tree problem.

Jarn\'{i}k and K\"{o}ssler \cite{Jarnik1934} first derived most of the known fundamental structural properties of Euclidean Steiner trees in 1934. The seminal work of Gilbert and Pollak \cite{gilpol} gave additional proofs of these properties and several others. Their structural results essentially remain the best existing tools for analyzing high-dimensional Euclidean Steiner trees.

\vspace{-0.4cm}\paragraph{Computational Aspects.}  Building on the work of Garey and Johnson \cite{Garey_Johnson_1977} wherein they proved that the Rectilinear Steiner Tree problem  (i.e., terminals are in $\ell_1$-metric space) is \np-hard, in a joint work with Graham \cite{Garey_Graham_Johnson_1977}, they proved that the Euclidean Steiner Tree problem is also \np-hard by a clever planar gadget construction. In their seminal works, Arora \cite{Arora_1998} and Mitchell \cite{M99} gave a polynomial-time approximation scheme (\ptas) for the Steiner Tree problem in all $\ell_p$-metric spaces, albeit in constant dimensions. However, their work left open the hardness of approximation of the Euclidean Steiner Tree problem  in high dimensions (such as when the dimension is at least logarithmic  in the number of terminals). Trevisan \cite{Trevisan00} showed that the Rectilinear   Steiner Tree problem is \apx-hard  by a reduction from the Steiner Tree problem  in the Hamming metric (which was previously shown to be \apx-hard \cite{Day_Johnson_Sankoff_1986}).  Trevisan's reduction appeals to the Hamming metric's discrete combinatorial structure. In fact, an even simpler proof can be derived from much earlier known structural results about Hamming and Rectilinear Steiner trees (e.g., Lemma 1 of \cite{Day_Johnson_Sankoff_1986} combined with Theorem 4 of \cite{Hanan_1966}).

Proving the \apx-hardness of the Steiner Tree problem in $\ell_p$-metrics, for $p>1$, appears to require engaging with the delicate structure of $\R^d$ directly. Due to the \ptas for the problem in fixed dimensions, any such argument requires dealing with truly high-dimensional hard instances. Recently, Fleischmann et al.\ \cite{Fleischmann2023} recently showed that the Steiner Tree problem is \apx-hard in the $\ell_{\infty}$-metric. They also showed that when the set of candidate Steiner points is provided as part of the input (as a special case of the graph Steiner Tree problem), then this discrete variant of the  Steiner Tree problem is \apx-hard in all $\ell_p$-metric spaces. However, the hardness of approximation of the classical Euclidean Steiner Tree problem remains unresolved.
\begin{open} \label{open: apx-hardness}
Is the Euclidean Steiner Tree problem \apx-hard in high dimensions?
\end{open}

 Existing techniques in the area paves the way for a simple approach to prove the \apx-hardness of the Euclidean Steiner Tree problem in high dimensions. Trevisan \cite{Trevisan00}
uses a simple gap-preserving reduction from the Vertex Cover problem, and a similar reduction applies in the setting where the candidate Steiner points are provided as input. 

As motivation, we sketch a simple reduction framework for proving \apx-hardness of the Euclidean Steiner Tree problem. We formalize this in Appendix \ref{sec: simplicial complex conj}. We reduce from the Vertex Cover problem on bounded degree triangle-free graphs. Namely, there exists some $\rho > 0$ such that it is \np-hard to decide whether a triangle-free graph $G=(V,E)$ has a vertex cover of size $r|E|$ or all of its vertex covers are of size at least $(1 + \rho)\cdot r|E|$ (for some $r,\rho>0$).    Now, we embed $G$ into $\R^{|V|}$ by embedding each edge $\{u,v\}\in E$ as $\mathbf{e}_u + \mathbf{e}_v$, where $\mathbf{e}_u$ is the standard basis vector with $1$ in coordinate indexed by $u$ and $0$ elsewhere. Thus, each edge is embedded as its characteristic vector. The embedding of the set of edges incident to a single vertex forms the regular simplex of side length $\sqrt{2}$. The point configuration as a whole composes of the vertices of a regular simplicial complex (where we take the union of the simplices associated with each vertex).

Observe that a vertex cover of $G$ of size $r|E|$ induces a partition of the embedded simplicial complex into $r|E|$ regular simplices. Moreover, any partition of the simplicial complex into regular simplices induces a vertex cover in $G$ of the same size (using that $G$ is triangle-free). Then, it would suffice to show that there exists $s, \beta > 0$ such that the following holds:
\begin{enumerate}
    \item Any embedded point configuration forming the vertices of a regular simplicial complex partitionable into $r|E|$ regular simplices has a Steiner tree of cost at most $s|E|$.
    \item Any embedded point configuration forming the vertices of a regular simpicial complex such that the minimum size of a partition into regular simplices is at least $(1 + \rho)r|E|$ has minimum Steiner tree of cost at least $(1 + \beta)s|E|$.
\end{enumerate}

Namely, this would imply that the Euclidean Steiner Tree problem is \np-hard to approximate within a factor less than $(1 + \beta)$. Why might we expect this reduction to even be gap-preserving? On the one hand, this reduction, interpreted instead in the $\ell_1$-metric \emph{is} used to show \apx-hardness of the Rectilinear Steiner Tree problem (e.g., see \cite{Trevisan00}). 

Additionally, heuristically, regular simplices have incredibly efficient Steiner trees, so having a valid Steiner tree composed of few Steiner trees of a regular simplex should result in especially low cost optimal Steiner trees. To understand this, we consider the notion of \textit{Steiner ratios}: the Steiner ratio of a finite point configuration  $P \subset \R^d$ is the ratio of the cost of its optimal Steiner tree to the cost of its minimum spanning tree.
For example, the Steiner ratio of the vertices of an equilateral triangle is $\sqrt{3}/2$—the optimal Steiner tree is formed by connecting the three vertices to a Steiner point at the center of the triangle. The Steiner ratio of a point configuration measures the efficiency of its Steiner tree relative to trivially connecting the points in a minimum spanning tree. Gilbert and Pollak famously conjectured that the vertices of an equilateral triangle, i.e., the vertices of a $2$-dimensional regular simplex, form the most efficient Steiner tree among all planar point configurations.
\begin{conjecture}[Gilbert-Pollak Steiner Ratio Conjecture \cite{gilpol}]
    The minimum Steiner ratio over planar point configurations is $\sqrt{3}/2$. 
\end{conjecture}
This important conjecture remains open after nearly 50 years (despite at least one high-profile incorrect proof \cite{Ivanov_Tuzhilin_2012}). They further conjectured that the vertices of a regular simplex have the minimum Steiner ratio in higher dimensions. This is false: for example, many regular simplices overlapping on a common vertex has a smaller Steiner ratio \cite{Du_Smith_1996}. Nonetheless, the constructions of all known counterexamples require increasing the number of points in the configuration.
We conjecture that this is necessary: the vertices of a regular simplex have the minimum Steiner ratio over all point configurations on at most that many terminals.

\begin{conjecture}[Simplex is the Best] \label{conj: simplex is the best}
    The $d+1 $ vertices of a $d$-dimensional regular simplex have the minimum Steiner ratio over all point configurations of $d+1$ points in Euclidean space.
\end{conjecture}

While this is akin to the generalized Gilbert-Pollak conjecture in that it is about point configurations minimizing the Steiner ratio, the key difference is that we bound the number of terminals, not the number of dimensions. This is much more natural from a computational perspective. The natural dimension bound is then that any $d+1$ points can be embedded into $d$-dimensional space. Importantly, regular simplices meet this bound.

This conjecture would have structural implications for our efforts to prove \apx-hardness of Euclidean Steiner Tree. The following weaker version of Conjecture \ref{conj: simplex is the best} is also relevant to this reduction strategy.
\begin{conjecture}[Simplex is the Best for Graph Embeddings] \label{conj: simplex is the best graph embed}
Over all graphs with $m$ edges, the embedding\footnote{Here we allude to embedding each edge by it's characteristic vector, as detailed in the aforementioned reduction from the Vertex Cover problem.} of the star graph on $m$ edges has the minimum cost Steiner tree.
\end{conjecture}
Note that, since the graphs all have the same number of edges, their minimum spanning tree costs are all the same. Hence, Conjecture \ref{conj: simplex is the best graph embed} can also be viewed as a conjecture about the point configuration with the minimum Steiner ratio. In this sense, it is a restricted version of Conjecture \ref{conj: simplex is the best}. We have verified the weaker Conjecture \ref{conj: simplex is the best graph embed} computationally up to $m = 10$ using the exact algorithm of Smith \cite{smith}.

The aforementioned reduction strategy to Euclidean Steiner tree problem appears deceptively simple to employ and either verify or reject. However, there is a fundamental obstacle: \textbf{we do not know how to construct the optimal Steiner tree for any  (non-trivial) high-dimensional Euclidean point configuration.} Perhaps the simplest possible point configuration in $\R^{d+1}$ is the collection of standard basis vectors. These points are precisely the vertices of a $d$-dimensional regular simplex. Even the optimal Steiner tree of the regular simplex is unknown, although Gilbert and Chung \cite{gilchung} constructed candidate optimal trees in the special case of the number of vertices being a sum of up to three powers of two. For almost every other natural high-dimensional point configuration, we are completely ignorant of the structure of the optimal Steiner tree. 

The objective of this paper is to revitalize this important line of work in the hope of ultimately resolving Open Question \ref{open: apx-hardness}. To achieve this, we need to extend our understanding of high-dimensional Euclidean Steiner trees beyond the results of the previous century.

\subsection{Organization of the paper}
In Section~\ref{sec: preliminaries}, we define the terminology we will use in discussing the Euclidean Steiner tree problem, establish notational conventions, and recall several relevant classical structural properties of Steiner trees. In Section~\ref{sec: new props}, we prove three new structural results about Euclidean Steiner trees: two of them extend previous results of \cite{gilpol} and the third provides a simple condition for restricting the topologies of optimal Steiner trees. In Section~\ref{sec: top conj}, we discuss a little-known conjecture of Smith \cite{smith} about the topology of optimal Steiner trees of the regular simplex, motivating it from a new viewpoint and describing its interdisciplinary connections to existing work in chemical graph theory and computational biology. In Section~\ref{sec: constructing trees}, we prove several new structural results about the optimal Steiner trees of the regular simplex and show how to explicitly construct Steiner trees of the conjectured optimal topologies. In Section~\ref{sec: simplex is the best graph embed}, we revisit Conjecture \ref{conj: simplex is the best graph embed} from Section~\ref{sec: intro}, making partial progress toward the conjecture. Finally, in Appendix~\ref{sec: simplicial complex conj}, we state an analytic conjecture about the the Steiner tree problem on regular simplicial complexes,  proving that it implies \apx-hardness of the Euclidean Steiner Tree problem (Open Problem \ref{open: apx-hardness}).

\section{Preliminaries} \label{sec: preliminaries}
In this paper, we consider the Euclidean Steiner Tree problem. The problem is as follows. Given $P \subset \R^d$ finite, find $S \subset \R^d$ such that the minimum cost spanning tree $T$ of $P \cup S$ has the infimum cost of all trees over all choices of $S$. The length of each edge in the tree is the Euclidean distance between its endpoints. The elements of $P$ are \textit{terminals}, the additional points in $S$ are \textit{Steiner points}, and any spanning tree of $P \cup S$ for any choice of $S$ is a \textit{Steiner tree}. An \textit{optimal Steiner tree} for $P$ is a Steiner tree of minimum cost for $P$. 

For much of the paper, we will consider Steiner trees of the vertices of \textit{regular simplices}. Regular simplices are polytopes such that all vertices are equidistant. In particular, the vertices of a $(d-1)$-dimensional regular simplex can be embedded in $d$ dimensions as the set of $d$ standard basis vectors in $\R^{d}$. For $1 \leq i \leq d$, we denote these vectors by $\mathbf{e_i}$. For clarity with the number of terminals, we use \textit{regular $d$-simplex} to refer to regular simplices with $d$ vertices. For simplicity, when we consider the Steiner trees of the regular $d$-simplex, we mean the vertices of the regular simplex expressed as standard basis vectors (unless otherwise specified).

We now recall several useful facts about optimal Steiner trees from \cite{gilpol}. Although it is always possible to trivially add Steiner points without increasing the length of a Steiner tree (by subdividing an edge), we assume that optimal Steiner trees do not contain such Steiner points.

\begin{theorem}[{\cite[\textsection 3.2]{gilpol}}]
\label{onetwenty}
    Let $\mathbf{x}$ be a vertex in an optimal Steiner tree. Suppose there are two edges incident to $\mathbf{x}$. Then the angle included by these edges is at least $120\degree$.
\end{theorem}

\begin{corollary}[{\cite[\textsection 3.3]{gilpol}}]
\label{coplanar}
     In an optimal Steiner tree, there are exactly 3 edges incident to every Steiner point. Moreover, these lines are co-planar.
\end{corollary}

The co-planarity property was not explicitly stated in \cite{gilpol}, but it is easy to see from Theorem \ref{onetwenty} and the first part of Corollary \ref{coplanar}.

\begin{theorem}[{\cite[\textsection 3.4]{gilpol}}] \label{thm: at most n-2 St pts}
In an optimal Steiner tree of $n$ terminals, there are at most $n - 2$ Steiner points. 
\end{theorem}

In particular, a Steiner tree with $n$ terminals and exactly $n-2$ Steiner points is a \textit{full} Steiner tree.

\begin{theorem}[{\cite[\textsection 3.5]{gilpol}}]
\label{convexhull}
    Let $P$ be the set of terminals of an optimal Steiner tree. Then all Steiner points lie within the convex hull of $P$.
\end{theorem}

\begin{theorem}[{\cite[\textsection 4, Uniqueness Theorem]{gilpol}}]
\label{unique}
    For a given topology of a Steiner tree on the set of terminals in any Euclidean space, there always exists a unique Steiner tree with this topology of minimal length.
\end{theorem}

By \textit{topology} we mean the choice of edges in the Steiner tree for a fixed number of (unlabeled) Steiner points. For a fixed topology, the \textit{relatively minimal tree} refers to the unique tree from Theorem \ref{unique}.

The easiest case to consider is the equilateral triangle, i.e. $2$-dimensional regular simplex. For general triangles, it was solved as early as $17$\textsuperscript{th} century by Torricelli, Cavalieri and others and it was thoroughly analysed from different points of view in the last century---see \cite{CourantRobbins41} or \cite{Spain96}. The conclusion is that the optimal Steiner tree contains a Steiner point (also called \emph{Fermat point}) if and only if all of the interior angles are strictly less than $120\degree$. In \cite{gilchung}, they provide a formula for computing all distances to the Steiner point when a suitable triangle is given.

\section{New structural properties of Steiner trees} \label{sec: new props}
In this section we introduce several structural results about Euclidean Steiner trees that apply even in the high-dimensional setting. Garey, Graham, and Johnson's proof of \np-hardness of Euclidean Steiner Tree proves \np-hardness in the plane using that edges do not cross in an optimal tree \cite{Garey_Graham_Johnson_1977}. This property is extremely powerful in analyzing planar Steiner trees, but, while still true in higher dimensions, it is no longer useful. We do not know how to prove even in \np-hardness of Euclidean Steiner Tree without proving it in the plane, partly for lack of effective higher dimensional tools. The purpose of this section is to begin to ameliorate this deficiency.

\subsection{Edge lengths in optimal Steiner trees}
In \cite[\textsection 8.4]{gilpol}, the authors presents a bound on the length of edges between Steiner points in any optimal Steiner tree relative to the nearby edges in the tree. The argument only applies to Steiner trees on the plane (e.g., see  \cite[\textsection 9]{lengauer2012combinatorial}). In this section we use the local planarity of optimal Steiner trees to prove a similar result which holds in any Euclidean space.

\begin{theorem}
\label{SQ6/2}
    Let $\mathbf{s_1}$ and $\mathbf{s_2}$ be two Steiner points connected by an edge in an optimal Steiner tree $T$ in the Euclidean space of dimension $d \geq 3$. Let $N(\mathbf{x})$ denote the neighbourhood of a vertex $\mathbf{x} \in T$ and  $L_{0} = \min\limits_{\mathbf{s} \, \in \{\mathbf{s_1}, \mathbf{s_2}\}} \min\limits_{\mathbf{v} \, \in N(\mathbf{s}) \setminus \{\mathbf{s_1}, \mathbf{s_2}\}} ||\mathbf{s}-\mathbf{v}||$. Then $||\mathbf{s_1} - \mathbf{s_2}|| \geq \left(\frac{\sqrt{6}}{2}-1\right)L_{0}$.
\end{theorem}

\begin{proof}
    The proof is similar to the proof in  \cite[\S8.4]{gilpol}. Without loss of generality, assume there is a point $\mathbf{x_1}$ in $T$ adjacent to $\mathbf{s_1}$ and at distance $L_0$ and let $\mathbf{y_1}$ be a point at distance $L_0$ from $\mathbf{s_1}$ in the line incident to $\mathbf{s_1}$ not containing $\mathbf{s_2}$ or $\mathbf{x_1}$. Similarly, let $\mathbf{x_2}, \mathbf{y_2}$ be the two points at distance $L_0$ from $\mathbf{s_2}$ on the lines of $T$ incident to $\mathbf{s_2}$ but not containing $\mathbf{s_1}$. See Figure \ref{fig: plane intersection}.
    

\begin{figure}[h]
    \centering

\tikzset{every picture/.style={line width=0.75pt}} 

\begin{tikzpicture}[x=0.75pt,y=0.75pt,yscale=-1,xscale=1]

\draw    (193.1,30.02) -- (459.91,62.32) ;
\draw  [fill={rgb, 255:red, 74; green, 74; blue, 74 }  ,fill opacity=1 ] (317.53,46.17) .. controls (317.53,43.69) and (319.54,41.68) .. (322.02,41.68) .. controls (324.5,41.68) and (326.5,43.69) .. (326.5,46.17) .. controls (326.5,48.65) and (324.5,50.66) .. (322.02,50.66) .. controls (319.54,50.66) and (317.53,48.65) .. (317.53,46.17) -- cycle ;
\draw  [draw opacity=0][fill={rgb, 255:red, 184; green, 233; blue, 134 }  ,fill opacity=1 ] (233.8,62.32) -- (497,62.32) -- (384.2,127.54) -- (121,127.54) -- cycle ;
\draw  [draw opacity=0][fill={rgb, 255:red, 80; green, 227; blue, 194 }  ,fill opacity=1 ][line width=0.75]  (321.09,162.98) -- (153,150.56) -- (171.31,95.57) -- (343.53,95.57) -- (321.09,162.98) -- cycle ;
\draw  [draw opacity=0][fill={rgb, 255:red, 184; green, 233; blue, 134 }  ,fill opacity=0.8 ] (160.65,128.18) -- (160.65,126.44) -- (170.21,95.57) -- (343.53,95.57) -- (333.43,128.18) -- (160.65,128.18) -- cycle ;
\draw  [color={rgb, 255:red, 0; green, 0; blue, 0 }  ,draw opacity=1 ][fill={rgb, 255:red, 189; green, 16; blue, 224 }  ,fill opacity=1 ] (342.25,95.57) .. controls (343.53,89.82) and (351.27,63.02) .. (352.51,56.6) .. controls (353.76,50.17) and (361.44,73.19) .. (360.16,73.83) .. controls (358.88,74.47) and (340.97,101.33) .. (342.25,95.57) -- cycle ;
\draw  [draw opacity=0][fill={rgb, 255:red, 80; green, 227; blue, 194 }  ,fill opacity=0.78 ][line width=0.75]  (193.12,26.51) -- (361.22,39.59) -- (343.53,95.57) -- (171.31,95.57) -- (193.12,26.51) -- cycle ;
\draw    (342.25,95.57) -- (358.88,140.33) ;
\draw  [fill={rgb, 255:red, 0; green, 0; blue, 0 }  ,fill opacity=1 ] (455.43,62.32) .. controls (455.43,59.84) and (457.43,57.83) .. (459.91,57.83) .. controls (462.39,57.83) and (464.4,59.84) .. (464.4,62.32) .. controls (464.4,64.8) and (462.39,66.81) .. (459.91,66.81) .. controls (457.43,66.81) and (455.43,64.8) .. (455.43,62.32) -- cycle ;
\draw  [fill={rgb, 255:red, 0; green, 0; blue, 0 }  ,fill opacity=1 ] (349.27,127.54) .. controls (349.27,125.06) and (351.28,123.05) .. (353.75,123.05) .. controls (356.23,123.05) and (358.24,125.06) .. (358.24,127.54) .. controls (358.24,130.02) and (356.23,132.04) .. (353.75,132.04) .. controls (351.28,132.04) and (349.27,130.02) .. (349.27,127.54) -- cycle ;
\draw    (459.91,62.32) -- (342.25,95.57) ;
\draw    (342.25,95.57) -- (211.8,95.57) ;
\draw  [fill={rgb, 255:red, 0; green, 0; blue, 0 }  ,fill opacity=1 ] (339.05,95.57) .. controls (339.05,93.09) and (341.05,91.08) .. (343.53,91.08) .. controls (346.01,91.08) and (348.02,93.09) .. (348.02,95.57) .. controls (348.02,98.05) and (346.01,100.06) .. (343.53,100.06) .. controls (341.05,100.06) and (339.05,98.05) .. (339.05,95.57) -- cycle ;
\draw  [fill={rgb, 255:red, 0; green, 0; blue, 0 }  ,fill opacity=1 ] (188.61,30.02) .. controls (188.61,27.54) and (190.62,25.53) .. (193.1,25.53) .. controls (195.57,25.53) and (197.58,27.54) .. (197.58,30.02) .. controls (197.58,32.5) and (195.57,34.51) .. (193.1,34.51) .. controls (190.62,34.51) and (188.61,32.5) .. (188.61,30.02) -- cycle ;
\draw  [fill={rgb, 255:red, 0; green, 0; blue, 0 }  ,fill opacity=1 ] (148.49,146.07) .. controls (148.49,143.59) and (150.5,141.58) .. (152.97,141.58) .. controls (155.45,141.58) and (157.46,143.59) .. (157.46,146.07) .. controls (157.46,148.55) and (155.45,150.56) .. (152.97,150.56) .. controls (150.5,150.56) and (148.49,148.55) .. (148.49,146.07) -- cycle ;
\draw    (173.44,128.17) -- (144.02,154.4) ;
\draw  [dash pattern={on 0.84pt off 2.51pt}]  (211.8,95.57) -- (173.44,128.17) ;
\draw    (188.78,17.56) -- (211.8,95.57) ;
\draw  [fill={rgb, 255:red, 0; green, 0; blue, 0 }  ,fill opacity=1 ] (207.32,95.57) .. controls (207.32,93.09) and (209.33,91.08) .. (211.8,91.08) .. controls (214.28,91.08) and (216.29,93.09) .. (216.29,95.57) .. controls (216.29,98.05) and (214.28,100.06) .. (211.8,100.06) .. controls (209.33,100.06) and (207.32,98.05) .. (207.32,95.57) -- cycle ;
\draw  [fill={rgb, 255:red, 74; green, 74; blue, 74 }  ,fill opacity=1 ] (272.54,95.57) .. controls (272.54,93.09) and (274.55,91.08) .. (277.03,91.08) .. controls (279.5,91.08) and (281.51,93.09) .. (281.51,95.57) .. controls (281.51,98.05) and (279.5,100.06) .. (277.03,100.06) .. controls (274.55,100.06) and (272.54,98.05) .. (272.54,95.57) -- cycle ;

\draw (365.45,55) node [anchor=north west][inner sep=0.75pt]   [align=left] {$\displaystyle \alpha $};
\draw (382.55,92.97) node [anchor=north west][inner sep=0.75pt]   [align=left] {$\displaystyle \Pi _{1}$};
\draw (291.74,140.29) node [anchor=north west][inner sep=0.75pt]   [align=left] {$\displaystyle \Pi _{2}$};
\draw (324.16,97.75) node [anchor=north west][inner sep=0.75pt]   [align=left] {$\displaystyle \mathbf{s_{1}}$};
\draw (216.73,97.75) node [anchor=north west][inner sep=0.75pt]   [align=left] {$\displaystyle \mathbf{s_{2}}$};
\draw (463.83,35) node [anchor=north west][inner sep=0.75pt]   [align=left] {$\displaystyle \mathbf{x_{1}}$};
\draw (164.56,15.59) node [anchor=north west][inner sep=0.75pt]   [align=left] {$\displaystyle \mathbf{x_{2}}$};
\draw (365.36,121.1) node [anchor=north west][inner sep=0.75pt]   [align=left] {$\displaystyle \mathbf{y_{1}}$};
\draw (158.18,152.42) node [anchor=north west][inner sep=0.75pt]   [align=left] {$\displaystyle \mathbf{y_{2}}$};
\draw (272.15,101.92) node [anchor=north west][inner sep=0.75pt]   [align=left] {$\displaystyle \mathbf{o}$};
\draw (323.31,16.23) node [anchor=north west][inner sep=0.75pt]   [align=left] {$\displaystyle \mathbf{c}$};

\end{tikzpicture}

    \caption{$3$-dimensional space defined by the intersecting planes with key points labeled.}
    \label{fig: plane intersection}
\end{figure}
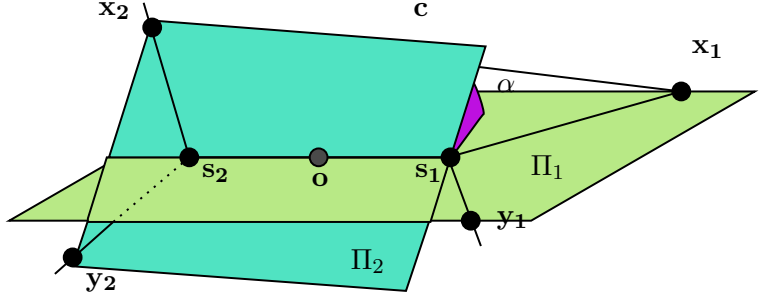
    
    From Corollary \ref{coplanar}, we have that for the Euclidean space of dimension $d \geq 2$, the lines incident to a given Steiner point are co-planar. Thus, two adjacent Steiner points define two planes that meet at a line, so we only need to consider the 3-dimensional space defined by these two intersecting planes. Let $\Pi_1$ and $\Pi_2$ denote such planes defined by the points $\mathbf{x_1}, \mathbf{y_1}, \mathbf{s_2}$ and $\mathbf{x_2}, \mathbf{y_2}, \mathbf{s_1}$, respectively. So either $\Pi_1$ and $\Pi_2$ meets only in the line defined by $\mathbf{s_1}$ and $\mathbf{s_2}$, or $\Pi_1 = \Pi_2$. We only need to consider the former case by  \cite[\S8.4]{gilpol}. For this, let $\alpha$ denote the angle at which $\Pi_1$ and $\Pi_2$ intersect ($\alpha$ is the acute angle between the lines $\overleftrightarrow{\mathbf{x_1y_1}}$ and $\overleftrightarrow{\mathbf{x_2y_2}}$ when both are translated to include the origin).

    Let us introduce a coordinate system with unit length $L_0$ and origin $\mathbf{o}$ placed at the midpoint at of the segment $\mathbf{s_1 s_2}$. If $x = ||\mathbf{x_1} - \mathbf{x_2}||$ and $s = ||\mathbf{s_1} - \mathbf{s_2}||$, then we have the following:
    
    \begin{align*}
        &\mathbf{o} = (0, 0, 0)\\
        &\mathbf{s_1} = \left(\frac{s}{2}, 0, 0\right)\\
        &\mathbf{s_2} = \left(-\frac{s}{2}, 0, 0\right)\\
        &\mathbf{x_1} = \left(\frac{1+s}{2}, \frac{\sqrt{3}}{2}\cos\left(-\frac{\alpha}{2}\right),  \frac{\sqrt{3}}{2}\sin\left(-\frac{\alpha}{2}\right)\right)\\
        &\mathbf{x_2} = \left(-\frac{1+s}{2}, \frac{\sqrt{3}}{2}\cos\left(\frac{\alpha}{2}\right), \frac{\sqrt{3}}{2}\sin\left(\frac{\alpha}{2}\right)\right)\\
        &\mathbf{y_1} = \left(\frac{1+s}{2}, \frac{-\sqrt{3}}{2}\cos\left(-\frac{\alpha}{2}\right),  \frac{-\sqrt{3}}{2}\sin\left(-\frac{\alpha}{2}\right)\right)\\
        &\mathbf{y_2} = \left(-\frac{1+s}{2}, \frac{-\sqrt{3}}{2}\cos\left(\frac{\alpha}{2}\right), \frac{-\sqrt{3}}{2}\sin\left(\frac{\alpha}{2}\right)\right)\\
        &\mathbf{c} = \left(0, \frac{\sqrt{3}}{2}\cos\left(\frac{\alpha}{2}\right), 0\right)
    \end{align*} 
    
    where $\mathbf{c}$ is the midpoint of the line segment $\mathbf{x_1}\mathbf{x_2}$.

    Now, we consider a second tree $T'$ derived
 from $T$ by removing $\mathbf{s_1}, \mathbf{s_2}$ and instead adding $\mathbf{x_1}, \mathbf{x_2}, \mathbf{y_1},$ and $\mathbf{y_2}$ as Steiner points and joining $\mathbf{x_1},\mathbf{x_2}$ (respectively, $\mathbf{y_1}, \mathbf{y_2}$) to the Fermat point of triangle $\triangle \mathbf{x_1}\mathbf{x_2}\mathbf{o}$ (respectively, $\triangle \mathbf{y_1}\mathbf{y_2}\mathbf{o}$), respectively, and connecting those Fermat points (via a line passing through $\mathbf{o}$). Call these Fermat points $\mathbf{r_1}$ and $\mathbf{r_2}$, respectively. We are interested in finding the coordinates of $\mathbf{r_1}$ and using the optimality of $T$ to obtain a bound on $\norm{\mathbf{s_1 - s_2}}$ (it suffices to consider $\mathbf{r_1}$ by symmetry in this coordinate system).

Consider $\triangle \mathbf{c}\mathbf{x_1}\mathbf{r_1}$. By symmetry of $\mathbf{x_1}$ and $\mathbf{x_2}$ about the line $\mathbf{o}\mathbf{c}$ and the definition of a Fermat point, $\triangle \mathbf{c}\mathbf{x_1}\mathbf{r_1}$ is a 30-60-90 triangle. Then, letting $x = \sqrt{(1+s)^{2}+3\sin^{2}(\frac{\alpha}{2})}$,  we have that $||\mathbf{c} - \mathbf{r_1}|| = \frac{x}{2\sqrt{3}}$, $\norm{\mathbf{c} - \mathbf{x_1}} = x/2$, and $\norm{\mathbf{x_1} - \mathbf{r_1}} = x/\sqrt{3}$. Then, the difference in the length of trees $T'$ and $T$ is
    
    \begin{align*}
        \underbrace{4||\mathbf{x_1} -  \mathbf{r_1}||+2(||\mathbf{o} - \mathbf{c}||-||\mathbf{c} - \mathbf{r_1}||)}_{\text{from } T} - \underbrace{(s+4)}_{\text{from } T'} =\sqrt{3}\left(x+\cos\left(\frac{\alpha}{2}\right)\right) - (s+4) \leq 0,
    \end{align*}

with the $\leq 0$ inequality coming from optimality of $T$. Solving for $s$, we get the inequality
    
    \begin{equation*}
        s \geq \sqrt{3}\cos\left(\frac{\alpha}{2}\right)-1.
    \end{equation*} 
    
    This is minimized for $\alpha = \frac{\pi}{2}$, where we get the inequality
    
    \begin{equation*}
    ||\mathbf{s_1} - \mathbf{s_2}|| \geq \left(\frac{\sqrt{6}}{2}-1\right)L_{0}.\qedhere
    \end{equation*}
\end{proof}

\subsection{Bounds on the coordinates of Steiner points}
We prove that optimal Steiner trees in Euclidean spaces are not only contained in the convex hull of the input terminals, but each of their coordinates are also strictly contained in the interval formed by the maximum and minimum value of the terminals in that coordinate.  This is formalized in the next lemma.

\begin{lemma} \label{lem: strict coordinate bounds}
Let $P \subseteq \R^d$ be a finite point-set such that for all $1\leq i \leq d$, there are $\mathbf{p}, \mathbf{q} \in P$ with $p_i \neq q_i$. Let $\mathbf{s} = (s_1, \ldots, s_d)$ be a Steiner point in an optimal Steiner tree of $P$. Then, for all $1 \leq i \leq d$ it holds that
\[
\min_{\mathbf{p} \in P} p_i < s_i < \max_{\mathbf{p} \in P}  p_i.
\]
\end{lemma}

\begin{proof}
Assume not. Without loss of generality, we may assume that the statement does not hold for $i=1$ and the bounds on the coordinate are $\min_{\mathbf{p} \in P} p_1 = 0$ and $\max_{\mathbf{p} \in P}  p_1 = 1$. We prove the lower bound and the upper bound follows analogously.

Let $T$ be an optimal Steiner tree for $P$. Suppose some Steiner point has first coordinate $0$. Observe that some neighbor of the Steiner point must also have first coordinate $0$ or else increasing the first coordinate of the Steiner point by an infinitesimal amount decreases the cost of the tree. 

Now, since some terminal has nonzero first coordinate, there exists some Steiner point $\mathbf{s}$ with first coordinate $0$ neighboring a point $\mathbf{x}$ with positive first coordinate. This uses the fact that $T$ is contained in the convex hull of $P$ by Theorem \ref{convexhull}. By the above, $\mathbf{s}$ also has a neighbor $\mathbf{y}$ with first coordinate $0$. Now recall Theorem \ref{onetwenty} and Corollary \ref{coplanar}: there are exactly three coplanar lines incident to $\mathbf{s}$. Any plane is defined by precisely two orthogonal lines. Using  $\overleftrightarrow{\mathbf{s}\mathbf{y}}$ as one of the two lines defining the plane, the other line must have non-fixed first coordinate. But we claim that the 120 degree angle property and coplanarity imply then that the 3\textsuperscript{rd} neighbor of $\mathbf{s}$, $\mathbf{z}$, has negative first coordinate. To see this, note that $\overleftrightarrow{\mathbf{s}\mathbf{y}}$ cuts the plane containing the lines incident to $\mathbf{s}$ into two parts. One side strictly contains $\mathbf{x}$ and the other strictly contains $\mathbf{z}$. Since the other line defining the plane has non-fixed first coordinate, this yields the claim.

However, we know from Theorem \ref{convexhull} that optimal Euclidean Steiner trees are contained in the convex hull of the terminal set, so this contradicts the optimality of $T$.
\end{proof}

In fact, we can apply this lemma to prove that all Steiner points must be strictly contained in the convex hull of the terminal configuration. Note that, as usual, this result assumes that our optimal Steiner trees do not contain trivial Steiner points (that is, Steiner points subdividing a line).
\begin{corollary}
For every finite $P \subseteq \R^d$, in any optimal Steiner tree $T$ of $P$,  all Steiner points in $T$ are strictly contained in the convex hull of $P$.
\end{corollary} \label{cor: strict containment in Convex Hull}
\begin{proof}
Suppose not. Suppose without loss of generality that the point configuration $P$ cannot be embedded in fewer than $d$ dimensions. Then, the convex hull of $P$ is the intersection of $(d-1)$-dimensional hyperplanes (and the Steiner tree $T$ is contained in the intersection of half-spaces). For each hyperplane composing part of the boundary of the convex hull, there is some terminal not contained in that hyperplane (or else the point configuration is embeddable in $\R^{d-1}$). Now, by rotating and translating the point configuration, we may assume that that hyperplane is a coordinate hyperplane corresponding to the first coordinate and that the interior of the convex hull is contained in the halfspace given by $\{\mathbf{x} \in \R^d \,:\, x_1 \geq 0\}$. Namely, $0$ is a lower bound on the first coordinate of each Steiner point and the upper bound is strictly greater than $0$ (since some terminal is not contained in the hyperplane). Then, by Lemma \ref{lem: strict coordinate bounds}, no Steiner point can lie on this coordinate hyperplane and, hence, no Steiner point could lie on the hyperplane before translation and rotation either.

The same procedure applies to all hyperplanes making up the convex hull of $P$ and hence the result follows.
\end{proof}

\subsection{Degree constraints on terminals}
We give a coordinate-based sufficient condition for a terminal being a 
leaf node in an optimal Steiner tree.
\begin{lemma} \label{lem: max pt prop}
Let $P \subseteq \R^d$. Suppose there is some point $\mathbf{p} = (p_1, \ldots, p_d) \in P$   such that for each $i \in [d]$ we have $p_i = \max\{q_i \,:\, \mathbf{q} \in P\}$ or $p_i = \min\{q_i \,:\, \mathbf{q} \in P\}$. Then, $\mathbf{p}$ is a leaf node in every optimal Steiner tree $T$ of $P$.
\end{lemma}
\begin{proof}
Let $T$ be a relatively minimal Steiner tree of $P$. Suppose that $\mathbf{p}$ has two lines incident to it with direction vectors $\vec{u}$ and $\vec{v}$ from $\mathbf{p}$ to the other endpoint of the edges, respectively. From Theorem \ref{convexhull}, it follows that the other endpoints must be contained in the convex hull of $P$ (whether the other endpoint be another terminal or a Steiner point). Therefore, in each coordinate $\vec{u}$ and $\vec{v}$ are either both non-negative or non-positive, depending on whether $\mathbf{p}$ is maximal or minimal in that coordinate.

Every pair of edges in $T$ sharing a node intersect at an angle of at least $120 \degree$. But, $\dotproduct{u}{v} \geq 0$ and thus they form an angle of at most $90 \degree$, contradicting the optimality of $T$, as desired.
\end{proof}

\section{The topology of Steiner trees of the regular simplex} \label{sec: top conj}
In this section, we consider a conjecture of Smith about the topology of the optimal Steiner tree of the regular simplex (Conjecture 2 of \cite{smith}). This conjecture extends the conjecture of \cite{gilchung} to regular simplices of all sizes. Smith verified the conjecture up to the regular $11$-simplex (by running Smith's algorithm, we confirm that it holds for the regular $12$-simplex as well). We provide new intuition for the conjecture and observe that trees with the conjectured topology have important extremal combinatorial properties with ties to chemical graph theory and the study of phylogenetic trees in computational biology.

\subsection{The conjectured topology}
\label{subsec: conjectured topology}
From Lemma \ref{lem: max pt prop}, we know that in any optimal Steiner tree of (the vertices of) a regular simplex, the leaf nodes are exactly the terminal nodes. Hence, there are exactly $n-2$ Steiner points (since each Steiner point is of degree precisely $3$).

Let $\mathbf{e_i}$ and $\mathbf{e_j}$ be a pair of terminals with a topology-preserving coordinate permutation $\sigma: [d] \to [d]$ between them in some optimal Steiner tree of the regular simplex, $T$. Consider rooting $T$ at $\mathbf{e_i}$. Then, we apply $\sigma$ to each level of $T$ rooted at $\mathbf{e_i}$ to get $T$ rooted at $\mathbf{e_i}$. Such topology-preserving coordinate permutations are hence very restrictive and lead to extensive symmetries in the Steiner trees. 

When does a topology-preserving coordinate permutation exist? Suppose that there exists an induced full binary subtree rooted at some Steiner point, with all its leaf nodes terminals. Then, the two subtrees induced by the children of the root can be interchanged by a permutation swapping the coordinates corresponding to the terminals in each subtree. Indeed, this observation may be repeated to yield a topology preserving coordinate permutation between any pair of terminals in the full subtree. We revisit this idea more rigorously in Section \ref{sec: constructing trees}.

 When $d = 2^k$, it is even possible to have topology-preserving coordinate permutation between every pair of terminals. One such topology permitting this is two full binary trees with $2^{k-1}$ terminals leaves each with their Steiner point root nodes connected by an edge (see Figure \ref{fig:rigid structure}). \input{Rigid_Structure}
 
 Given the rigid structure induced by topology-preserving coordinate permutations, it feels plausible that trees exhibiting large full binary subtree structures are extremal. The topologies of the optimal Steiner trees (derived computationally via Smith's algorithm) are shown in Figure \ref{fig: optimal topologies}. 

 \begin{figure}[h!]
    \centering

\input{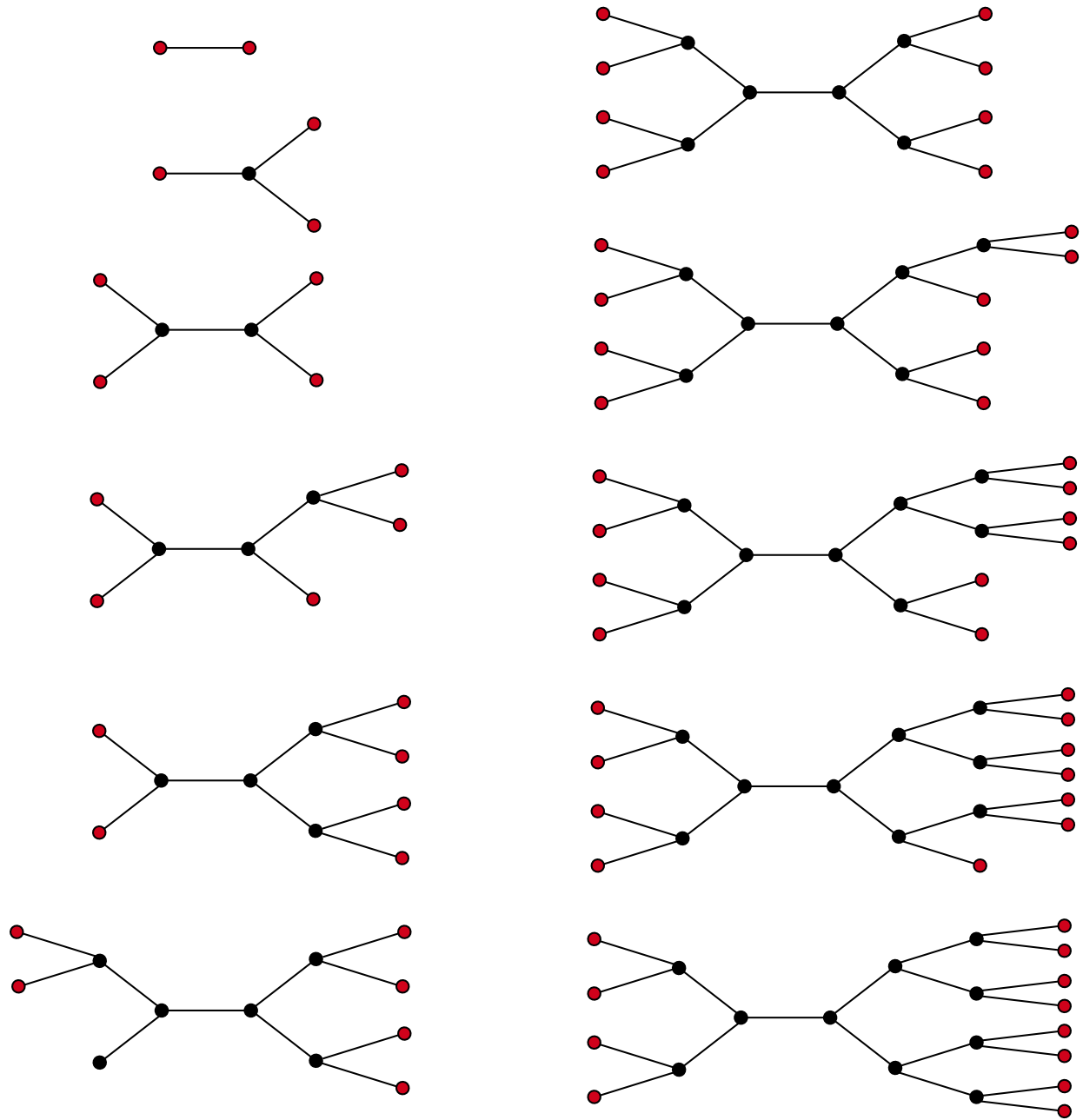}

    \caption{The optimal topology of the Euclidean Steiner tree of the regular simplex with up to $12$ vertices, terminals marked red.}
    \label{fig: optimal topologies}
\end{figure}

 These topologies interpolate between the pair of full binary trees topologies for number of terminals a power of $2$. We make this notion rigorous. We define a \textit{good} binary tree of height $k$ recursively. For height $0$, it is a single node. For height $k > 0$, at most one subtree of the children of the root is a \textit{good} binary tree of height $k-1$ and the remaining subtrees are full binary trees of height $k-1$ or $k-2$.  Observe that in any good binary tree of height $k$ the shortest distance from a root to a leaf node must be at least $k-1$ (this follows by a simple inductive argument). Hence, a good binary tree of height $k$ contains strictly more than $2^{k-1}$ leaf nodes (and at most $2^k$ leaf nodes).

 \begin{claim}
For $k \geq 1$ and $2^k < d \leq 2^{k+1}$ leaf nodes,  there is a unique good tree of height $k + 1$ (up to isomorphism) with that number of leaf nodes.
 \end{claim}
 \begin{proof}
    We prove this by induction. This is clear for $d = 2$ leaves: the tree must be height $1$ and hence must be a full binary tree of height $1$. 

    Now, assume this holds for all $d < r$ and assume $2^k < r \leq 2^{k+1}$. Then, if $r \leq 2^k + 2^{k-1}$, by definition of a good binary tree of height $k + 1$ and the fact that any good tree of height $k$ contains strictly more than $2^{k-1}$ leaf nodes, one of the subtrees of the children of the root must be a full binary tree of height $k-1$. Then the other subtree must be a good subtree of height $k - 1$ on $r - 2^{k-1}$ leaves, which, by induction, is unique up to isomorphism.

    If $r > 2^k + 2^{k-1}$, since a good tree of height $k$ contains at most $2^k$ leaf nodes, one of the subtrees of the children of the root must be a full binary tree of height $k$. Then, the other subtree must be a good subtree of height $k-1$ on $r - 2^k$ leaves, yielding the desired via the inductive hypothesis.
 \end{proof}

Finally, we can state the conjecture, a reformulated version of a conjecture of Smith \cite{smith}.
 \begin{conjecture}[Optimal Topology of Steiner Trees of the Regular Simplex, Conjecture 2 of \cite{smith}] \label{conj: top of regular simplex}
Let $2 \leq 2^k < d \leq 2^{k+1}$. The topology of the optimal Steiner tree of a regular $d$-simplex is formed by taking the good tree of height $k+1$ on $d$ leaf nodes, removing the root node, and reconnecting the tree via an edge between the former children of the root.
 \end{conjecture}

In the next section we provide additional motivation for this topology, remarking that these trees are extremal with respect to a well-studied index of acyclic graphs from chemical graph theory.

\subsection{Indices from chemical graph theory}
Throughout this subsection, we will consider trees of the form of full Steiner trees: that is, trees with all non-leaf nodes of degree exactly three. Our connection to computational biology is simple: phylogenetic trees are precisely trees with this structure. We begin by defining several notions from chemical graph theory. The \textit{Wiener index} \cite{Wiener_1947} of a tree is the sum of pairwise hop-distances between nodes in the tree. I.e., for a tree $T$,
\[
W(T) = \sum_{u,v \in V(T)} d_T(u,v).
\]
Since its introduction in 1947, this has been one of the most widely used metrics in the study of quantitative structure-activity relationships in chemistry. See \cite{dobrynin2001wiener} for a thorough survey. 

A related, more recently introduced index is the \textit{terminal Wiener index} \cite{Gutman_Furtula_Petrovic_2009} of a tree, the Wiener index restricted only to pairs of leaf nodes. This index has garnered significant attention since its introduction \cite{Deng_Zhang_2012, Cheng_Zhang_2013, Humphries_Wu_2013, Zeryouh_Marraki_Essalih_2014, Zeryouh_El_Marraki_Essalih_2016, Ramane_Bhajantri_Kitturmath_2021, Ramane_Bhajantri_Kitturmath_2021, nayaki2022diagrammatic, nayaki2022physical, Sulphikar_2022}. Formally, let $T$ be a tree and $H$ an induced subgraph of $T$. Then, let $\ell(H)$ be the set of leaf nodes from $T$ contained in $H$. For any edge $(u,v) \in E(T)$, let $(C_u, C_v)$ be the cut induced by the edge. We have that
\[
\Gamma(T) = \sum_{u, v \in \ell(T)} d_T(u,v) = \sum_{(u,v) \in E(T)} |\ell(C_u)| |\ell(C_v)|.
\]

This brings us to our main connection to these indices.
\begin{theorem}[\cite{Humphries_Wu_2013, Szekely2011}]
The trees of the form described in Conjecture \ref{conj: top of regular simplex} are the unique full Steiner trees minimizing the terminal Wiener index.
\end{theorem}
Indeed, \cite{Humphries_Wu_2013} reveals several other interesting properties of this extremal topology. 
\begin{definition}[Semi-regularity]
Given a full Steiner tree $T$ and some pair $u,v$ of Steiner points, let $T_u^1$ and $T_u^2$ denote the subtrees rooted at the children of $u$ upon removing the path $u \leadsto v$.  Similarly define $T_v^1$ and $T_v^2$. For $H$ an induced subtree of $T$, let $P(H)$ denote set of terminals in $H$. The pair $(u,v)$ is \textit{semi-regular} if 
\[
\min(|P(T_v^1)|,|P(T_v^2)|) \geq \max(|P(T_u^1)|,|P(T_u^2)|)
\]
or
\[
\min(|P(T_u^1)|,|P(T_u^2)|) \geq \max(|P(T_v^1)|,|P(T_v^2)|).
\]
The tree $T$ is \textit{semi-regular} if every pair of Steiner points is semi-regular.
\end{definition}
If a pair were not semi-regular, then swapping a larger and a smaller subtree would result in a more balanced tree. Intuitively, this should not occur in optimal Steiner trees of the regular simplex due to its myriad symmetries. Indeed, we have the following.
\begin{theorem}[\cite{Humphries_Wu_2013}]
The unique semi-regular full Steiner trees are exactly the trees of the form described in Conjecture \ref{conj: top of regular simplex}.
\end{theorem}

\section{Steiner trees of the regular simplex} \label{sec: constructing trees}
In this section we explicitly compute Steiner trees of the regular simplex. Our constructions are a stronger version of the results of \cite{gilchung}. In the first subsection, we show several useful properties about Steiner trees of the regular simplex. In the second subsection, we describe our candidate construction for optimal Steiner trees. In the third subsection, we will apply this construction to give explicit coordinates for the candidate-optimal Steiner trees on $d = 2^k$ terminals and use these explicit coordinates to analyze the limiting Steiner ratio of the regular simplex.

\subsection{Structural properties of Steiner trees of the regular simplex}
The following definition will be useful.

\begin{definition}[Extending Line]
    Let $T$ be a tree in the Euclidean space. Let $\mathbf{x}, \mathbf{y} \in V(T)$ be two adjacent points in $T$. The \emph{extending line} of the edge $(\mathbf{x}, \, \mathbf{y})$ is the line containing the segment corresponding to the edge $(\mathbf{x}, \, \mathbf{y})$.
\end{definition}

Applying Theorem \ref{unique} for Steiner trees of the regular simplex, we give the following lemma.
\begin{lemma}
\label{symmetries}
    Let $T$ be a relatively minimal Steiner tree of the regular $n$-simplex. Let $T'$ be an induced full binary subtree of $T$, with all leaf nodes of $T'$ terminals. Let $P(T')$ be the set of terminals in $T'$ and $\mathbf{r_{T'}}$ the root Steiner point of $T'$. Let $e_{T'}$ be the edge incident to $\mathbf{r_{T'}}$ that does not lie in $T'$. Then the following hold:
    \begin{enumerate}
        \item For all Steiner points $\mathbf{s}$ in $(T \setminus T') \cup \{ \mathbf{r_{T'}}\}$, for all coordinates $i, j$ such that terminals $\mathbf{e_i}, \mathbf{e_j} \in P(T')$, it holds that $s_i=s_j$.
         \item Let $\ell_{T'}$ be the line extending the edge $e_{T'}$. Then $\ell_{T'}$ passes through the centroid of the terminals in $P(T')$.
    \end{enumerate}
\end{lemma}

The proof of this lemma uses permutations of the labels of the terminals, as informally sketched in Section \ref{subsec: conjectured topology}. We formalize that notion here using labeled full binary trees.

\begin{definition}[Labelling Full Binary Trees]
\label{ATB}
    Let $T$ be a full binary tree (with a fixed choice of left and right children for every non-leaf node). Then, \textit{labelling the tree with respect to (a binary string) $g$} is defined as follows. 
    \begin{itemize}
        \item The root of $T$ is labeled as $g$.
        \item Then, while there exists some labeled node $v$ with unlabeled children, label the left child of $v$ by appending $0$ to the label of $v$ and label the right child of $v$ by appending $1$ to the the label of $v$. E.g., if $v$ had binary string $b$ as its label, its children will be labeled $b0$ and $b1$.
    \end{itemize}
    We denote $T_g$ as the result of labelling $T$ with respect to $g$.
\end{definition}

    Let $T$ be a Steiner tree of the regular simplex and let $T_g'$ be a labeled induced full binary subtree with terminal leaf nodes (with an arbitrary choice of left and right children). Let $\mathbf{s}$ be a Steiner point in $T_g'$. Consider the subtree of $T_g'$ rooted at $\mathbf{s}$. Each node in the left subtree has its label prepended by $b_{\mathbf{s}}0$ and each node in the right subtree has its label prepended by $b_{\mathbf{s}}1$. Swapping the left and right children of $\mathbf{s}$ amounts to making the labels of the nodes in the left subtree prepended by $b_{\mathbf{s}}1$ and the labels of the nodes in the right subtree instead prepended by $b_{\mathbf{s}}0$ (via Definition \ref{ATB}).

\begin{observation}
\label{labels}
      Let $T$ be a Steiner tree of the regular simplex and let $T_g'$ be a labeled induced full binary subtree with terminal leaf nodes. Let $\mathbf{s}$ be a Steiner point in $T_g'$. Swapping the left and right children of $\mathbf{s}$ corresponds to a topology-preserving coordinate permutation in $T$.
\end{observation}

\begin{proof}
    The coordinate permutation is precisely the composition of the involutions of each pair of terminals whose labels are $b_{\mathbf{s}}0c$ and $b_{\mathbf{s}}1c$. Namely, their labels differ exactly in the bit after the label of $\mathbf{s}$. 
    
    Since each pair of swapped terminals is between terminals in the corresponding position in the other subtree of $\mathbf{s}$ (and the subtrees have the same overall structure since $T_g'$ is an induced full binary subtree with terminal leaf nodes), this permutation of the terminals preserves the topology of $T$.
\end{proof}

Finally, we can return to Lemma \ref{symmetries}.

\begin{proof}
    First, label $T'$ with respect to the empty string $\varepsilon$ to obtain $T'_{\varepsilon}$. We claim that for every pair of terminals $\mathbf{e_i}, \mathbf{e_j} \in T'$ there is a topology-preserving coordinate permutation of $T$ that swaps $\mathbf{e_i}$ and $\mathbf{e_j}$. Namely, it is the composition of topology-preserving coordinate permutations induced by swapping the left and right children of Steiner points in $T'_{\varepsilon}$ (via Observation \ref{labels}). Index the bits of the labels of the nodes from last added to first (via Definition \ref{ATB}). Suppose that the labels of $\mathbf{e_i}$ and $\mathbf{e_j}$ differ on bits $a_1, a_2, \ldots, a_{r}$. Swapping the left and right children of each of the $a_1$\textsuperscript{th}, $a_2$\textsuperscript{th}, \ldots, $a_{r}$\textsuperscript{th} ancestors of $\mathbf{e_i}$ and $\mathbf{e_j}$ (where the first ancestor is the parent of $\mathbf{e_i}$, the second ancestor is the grandparent, etc., and we only swap the children of a node at most once) swaps the labels of $\mathbf{e_i}$ and $\mathbf{e_j}$ (for example see Figure \ref{fig: swapping children}). The composition of the corresponding topology-preserving coordinate permutations is a topology-preserving coordinate permutation swapping $\mathbf{e_i}$ and $\mathbf{e_j}$. But, by Theorem \ref{unique}, $T$ must be fixed (up to isomorphism) under this map.
 
    \input{Swapping_Children}
    
    The topology of $(T \setminus T') \cup \{ \mathbf{r_{T'}}\}$ is fixed under this topology-preserving coordinate permutation and each terminal is fixed. Hence, the Steiner points in this part of the tree must be fixed (or else we would violate Theorem \ref{unique}). Hence, for each such Steiner point $\mathbf{s}$, we must have $s_i = s_j$. This holds for all $\mathbf{e_i}, \mathbf{e_j} \in T'$, yielding the first part of the result.

     Observe that $e_{T'}$ is an edge between two Steiner points such that the first property holds. Hence, for every point on the extending line $\ell_{T'}$, $s_i = s_j$ for all pairs of terminals $\mathbf{e_i}, \mathbf{e_j} \in T'$. Using this observation, we prove the second part of the result by induction on the number of levels in the induced full binary subtree $T'$. When there is only one level, $T'$ is a Steiner point $\mathbf{r_{T'}}$ connected to two terminals, $\mathbf{e_i}$ and $\mathbf{e_j}$. By Corollary \ref{coplanar}, the edge $e_{T'}$ and the edges from $\mathbf{r_{T'}}$ to $\mathbf{e_i}$ and $\mathbf{e_j}$ are all coplanar. Consider $\triangle \mathbf{r_{T'}}\mathbf{e_i}\mathbf{e_j}$. By coplanarity, $\ell_{T'}$ lies in the same plane as this triangle. By the observation about the extending line above, $\ell_{T'}$ must be coincident with the perpendicular bisector of the line segment joining $\mathbf{e_i}$ and $\mathbf{e_j}$. Hence, $\ell_{T'}$ passes through the midpoint of this segment, the centroid of $\mathbf{e_i}$ and $\mathbf{e_j}$.

     The inductive step is similar: suppose the result holds for full binary subtrees with at least $r$ levels. Then, consider the edges other than $e_{T'}$ incident to the root $\mathbf{r_{T'}}$ of $T'$. Namely, with $T_L$ the left subtree of $\mathbf{r_{T'}}$ and $T_R$ similarly defined, these edges are precisely $e_{T_L}$ and $e_{T_R}$. By the inductive hypothesis, $\ell_{T_L}$ and $\ell_{T_R}$ pass through $\mathbf{c_{T_R}}$ and $\mathbf{c_{T_L}}$, the centroids of $P(T_R)$ and $P(T_L)$, respectively. Now, consider $\triangle \mathbf{r_{T'}} \mathbf{c_{T_R}} \mathbf{c_{T_L}}$. Again, coplanarity and the observation about the extending line imply that $\ell_{T'}$ is coincident to the perpindicular bisector of the line segment joining $\mathbf{c_{T_R}}$ and  $\mathbf{c_{T_L}}$. Hence, it passes through the midpoint of $\mathbf{c_{T_R}}$ and $\mathbf{c_{T_L}}$, the centroid of $P(T')$, yielding the desired result.    
\end{proof}

We need to say that the terminals of an optimal Steiner tree for the $n$-simplex have to be its leaf nodes. Luckily, this follows from a previous structural lemma.

\begin{corollary}
\label{leaves}
In any relatively minimal Steiner tree of the regular $n$-simplex, all terminals are leaf nodes.
\end{corollary}
\begin{proof}
Consider the regular $n$-simplex with terminals $\mathbf{p_i} = \mathbf{e_i}$. The result then follows from Lemma \ref{lem: max pt prop}.
\end{proof}

Next, we will restrict the intersection of the extending lines of edges incident to Steiner points.
\begin{lemma}
\label{intersectingfaces}
    Let $F_i = \{\mathbf{x} \in \mathbb{R}^{n}: x_{i} = 0, \norm{\mathbf{x}}_1 = 1\}$ be a face of the regular $n$-simplex. Let $T$ be a Steiner tree of the regular $n$-simplex. Let $\mathbf{s} \in V(T)$ be a Steiner point. Let $\mathbf{a}, \mathbf{b},$ and $\mathbf{c}$ be the points of intersection between the convex hull and the rays extending the edges incident to $\mathbf{s}$, with the endpoint of the first ray the neighbor of $\mathbf{s}$ and the endpoints of the other rays $\mathbf{s}$. If $\mathbf{a} \not\in F_i$,  then at least one of $\mathbf{b}$ and $\mathbf{c}$ are also not contained in $F_i$.
\end{lemma}

\begin{proof}
     For the sake of contradiction, suppose that both $\mathbf{b}, \mathbf{c} \in F_i$. Let $\mathbf{a} \in F_j$ for $i \neq j$.  By Corollary \ref{coplanar} $\mathbf{a}, \mathbf{b}, \mathbf{c}, \mathbf{s}$ are coplanar. Hence, the lines $\overleftrightarrow{\mathbf{b}\mathbf{c}}$ and $\overleftrightarrow{\mathbf{s}\mathbf{a}}$ intersect at some point $\mathbf{x}$. Note that $\mathbf{s}\mathbf{a}$ includes an angle of exactly $60\degree$ with each of $\mathbf{s}\mathbf{b}$ and $\mathbf{s}\mathbf{c}$ (by Theorem \ref{onetwenty}). So, in particular, coplanarity implies that $\mathbf{b}\mathbf{c}$ intersects the ray $\overrightarrow{\mathbf{s}\mathbf{a}}$. Finally, all the points on $\mathbf{b}\mathbf{c}$ are contained in the convex hull of the simplex, so $\mathbf{x}$ must be in the convex hull as well. By Theorem \ref{convexhull}, $\mathbf{s}$ is in the convex hull of the simplex, so all the points in the convex hull on the ray $\overrightarrow{\mathbf{s}\mathbf{a}}$ are contained in the segment $\mathbf{s}\mathbf{a}$. Namely, $\mathbf{x}$ must be contained in the segments $\mathbf{b}\mathbf{c}$ and $\mathbf{s}\mathbf{a}$ (see Figure \ref{fig:intersecting faces}).
     
     Now, since $b_i = c_i = 0$, we have $x_i = 0$. But, $s_i, a_i > 0$ (using Lemma \ref{lem: strict coordinate bounds}), so we also have $x_i > 0$, yielding the desired contradiction.
\end{proof}

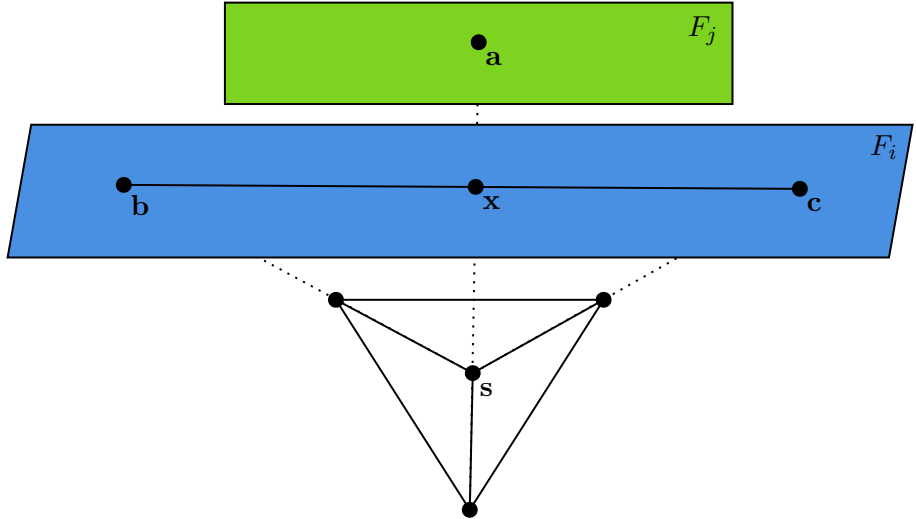
\begin{figure}[h]
    \centering

\tikzset{every picture/.style={line width=0.75pt}} 

\begin{tikzpicture}[x=0.75pt,y=0.75pt,yscale=-1,xscale=1]

\draw  [draw opacity=0] (329.5,257) -- (262,151) -- (397,151) -- cycle ;
\draw    (331,188) -- (329.5,257) ;
\draw    (262,151) -- (331,188) ;
\draw    (397,151) -- (331,188) ;
\draw  [dash pattern={on 0.84pt off 2.51pt}]  (496,95) -- (331,188) ;
\draw  [dash pattern={on 0.84pt off 2.51pt}]  (334,21) -- (329.5,257) ;
\draw  [dash pattern={on 0.84pt off 2.51pt}]  (155,93) -- (331,188) ;
\draw  [draw opacity=0][fill={rgb, 255:red, 126; green, 211; blue, 33 }  ,fill opacity=0.63 ] (206,1) -- (462,1) -- (462,52.2) -- (206,52.2) -- cycle ;
\draw  [fill={rgb, 255:red, 0; green, 0; blue, 0 }  ,fill opacity=1 ] (331.22,18.86) .. controls (332.41,17.32) and (334.61,17.04) .. (336.15,18.22) .. controls (337.68,19.4) and (337.96,21.61) .. (336.78,23.14) .. controls (335.59,24.68) and (333.39,24.96) .. (331.85,23.78) .. controls (330.32,22.6) and (330.04,20.39) .. (331.22,18.86) -- cycle ;
\draw  [fill={rgb, 255:red, 0; green, 0; blue, 0 }  ,fill opacity=1 ] (328.22,185.86) .. controls (329.41,184.32) and (331.61,184.04) .. (333.15,185.22) .. controls (334.68,186.4) and (334.96,188.61) .. (333.78,190.14) .. controls (332.59,191.68) and (330.39,191.96) .. (328.85,190.78) .. controls (327.32,189.6) and (327.04,187.39) .. (328.22,185.86) -- cycle ;
\draw  [fill={rgb, 255:red, 0; green, 0; blue, 0 }  ,fill opacity=1 ] (394.22,148.86) .. controls (395.41,147.32) and (397.61,147.04) .. (399.15,148.22) .. controls (400.68,149.4) and (400.96,151.61) .. (399.78,153.14) .. controls (398.59,154.68) and (396.39,154.96) .. (394.85,153.78) .. controls (393.32,152.6) and (393.04,150.39) .. (394.22,148.86) -- cycle ;
\draw  [fill={rgb, 255:red, 0; green, 0; blue, 0 }  ,fill opacity=1 ] (259.22,148.86) .. controls (260.41,147.32) and (262.61,147.04) .. (264.15,148.22) .. controls (265.68,149.4) and (265.96,151.61) .. (264.78,153.14) .. controls (263.59,154.68) and (261.39,154.96) .. (259.85,153.78) .. controls (258.32,152.6) and (258.04,150.39) .. (259.22,148.86) -- cycle ;
\draw  [fill={rgb, 255:red, 0; green, 0; blue, 0 }  ,fill opacity=1 ] (326.72,254.86) .. controls (327.91,253.32) and (330.11,253.04) .. (331.65,254.22) .. controls (333.18,255.4) and (333.46,257.61) .. (332.28,259.14) .. controls (331.09,260.68) and (328.89,260.96) .. (327.35,259.78) .. controls (325.82,258.6) and (325.54,256.39) .. (326.72,254.86) -- cycle ;
\draw  [draw opacity=0][fill={rgb, 255:red, 74; green, 144; blue, 226 }  ,fill opacity=0.64 ] (108.3,62.64) -- (552.8,62.64) -- (540.92,129.64) -- (96.42,129.64) -- cycle ;
\draw  [fill={rgb, 255:red, 0; green, 0; blue, 0 }  ,fill opacity=1 ] (493.22,92.86) .. controls (494.41,91.32) and (496.61,91.04) .. (498.15,92.22) .. controls (499.68,93.4) and (499.96,95.61) .. (498.78,97.14) .. controls (497.59,98.68) and (495.39,98.96) .. (493.85,97.78) .. controls (492.32,96.6) and (492.04,94.39) .. (493.22,92.86) -- cycle ;
\draw  [fill={rgb, 255:red, 0; green, 0; blue, 0 }  ,fill opacity=1 ] (152.22,90.86) .. controls (153.41,89.32) and (155.61,89.04) .. (157.15,90.22) .. controls (158.68,91.4) and (158.96,93.61) .. (157.78,95.14) .. controls (156.59,96.68) and (154.39,96.96) .. (152.85,95.78) .. controls (151.32,94.6) and (151.04,92.39) .. (152.22,90.86) -- cycle ;
\draw    (155,93) -- (496,95) ;
\draw  [fill={rgb, 255:red, 0; green, 0; blue, 0 }  ,fill opacity=1 ] (329.72,91.86) .. controls (330.91,90.32) and (333.11,90.04) .. (334.65,91.22) .. controls (336.18,92.4) and (336.46,94.61) .. (335.28,96.14) .. controls (334.09,97.68) and (331.89,97.96) .. (330.35,96.78) .. controls (328.82,95.6) and (328.54,93.39) .. (329.72,91.86) -- cycle ;

\draw (530,66.4) node [anchor=north west][inner sep=0.75pt]    {$F_{i}$};
\draw (438,6.4) node [anchor=north west][inner sep=0.75pt]    {$F_{j}$};
\draw (333,191.4) node [anchor=north west][inner sep=0.75pt]    {$\mathbf{s}$};
\draw (334.5,97.4) node [anchor=north west][inner sep=0.75pt]    {$\mathbf{x}$};
\draw (336,24.4) node [anchor=north west][inner sep=0.75pt]    {$\mathbf{a}$};
\draw (498,98.4) node [anchor=north west][inner sep=0.75pt]    {$\mathbf{c}$};
\draw (157,96.4) node [anchor=north west][inner sep=0.75pt]    {$\mathbf{b}$};

\end{tikzpicture}

    \caption{The intersection of the line segments $\mathbf{bc}$ and $\mathbf{sa}$ is the point $\mathbf{x}$.}
    \label{fig:intersecting faces}
\end{figure}

\begin{corollary} \label{cor: finite path}
Let $T$ be an optimal Steiner tree of the regular simplex. Let $\mathbf{s}$ be a Steiner point adjacent to a terminal $\mathbf{e_i}$ in $T$. Then, the rays extending the other edges of $\mathbf{s}$, both with endpoints $\mathbf{s}$, intersect the convex hull of the regular simplex at the face $F_i = \{\mathbf{x} \in \mathbb{R}^{n}: x_{i} = 0, \norm{\mathbf{x}}_1 = 1\}$.
\end{corollary}
\begin{proof}
 It holds that in every full, optimal Steiner tree, there are exactly $n-2$ Steiner points. It will suffice to use Lemma \ref{intersectingfaces} until we run out of possible Steiner points. Let $\mathbf{a}$ be a point of intersection between a ray extending an edge of $\mathbf{s}$ (not to $\mathbf{e_i}$) with endpoint $\mathbf{s}$. For the sake of contradiction, assume that $\mathbf{a} \not\in F_i$.
    
Let $\mathbf{s_0}=\mathbf{s}$ and let us proceed by induction on $k$. Let $e_k=(\mathbf{s_{k-1}}, \, \mathbf{s_k})$ be the edge such that the ray with endpoint $\mathbf{s_{k-1}}$ extending the edge $e_k$ does not intersect $F_i$ (e.g., $e_1$ is the edge whose extending ray has endpoint $\mathbf{a}$). If $\mathbf{s_k}$ is a terminal, since the $i$\textsuperscript{th} coordinate is greater than $0$, $\mathbf{s_i} = \mathbf{e_i}$ and we found a cycle in our tree $T$, a contradiction. Otherwise, we iterate using Lemma \ref{intersectingfaces} over up to all $n-2$ Steiner points in a path. Eventually then $\mathbf{s_i}$ must be a terminal since there are finitely many Steiner points, yielding a contradiction.    
\end{proof}

Finally, we are prepared to prove the following theorem.

\begin{theorem}
\label{orphanlength}
    Let $n \in \mathbb{N}$ and let $T$ be an optimal Steiner tree for the regular $n$-simplex. Consider a terminal $\mathbf{p} \in T$ and let $\mathbf{s}$ be the Steiner point adjacent to $\mathbf{p}$. Then, $\norm{\mathbf{s}-\mathbf{p}} > \frac{1}{\sqrt{3}}$.
\end{theorem}

\begin{proof}
     Without loss of generality, suppose $\mathbf{p} = \mathbf{e_1}$. By Corollary \ref{cor: strict containment in Convex Hull}, we may assume that $\mathbf{s}$ is strictly contained in the convex hull of the $n$-simplex. Now consider the two other edges incident to $\mathbf{s}$. The rays with endpoint at $\mathbf{s}$ extending these edges must intersect the convex hull of the $n$-simplex in some points $\mathbf{a}$, $\mathbf{b}$. Let $F_1 = \{\mathbf{x} \in \mathbb{R}^{n}: x_{1} = 0, \norm{\mathbf{x}}_1 = 1\}$. 
     
     By Corollary \ref{cor: finite path}, we know that  $\mathbf{a}$, $\mathbf{b}$ lie in $F_1$. We will show that  $\norm{\mathbf{s}-\mathbf{p}} > \frac{1}{\sqrt{3}}$.
 Since $\mathbf{a}, \mathbf{b} \in F_1$, it holds that 
 \begin{align*}
 \mathbf{a} &=\left(0, a_{2}, \ldots, a_{n}\right), \\
 \mathbf{b} & =\left(0, b_{2}, \ldots, b_{n}\right),
 \end{align*}
where $\sum_{i=2}^{n} a_{i} = \sum_{i=2}^{n} b_{i} = 1$.

    Now consider $\triangle \mathbf{p}\mathbf{a}\mathbf{b}$. It has a unique Fermat point and that must be $\mathbf{s}$. 
    We will use a calculation based on \cite{gilchung} to determine the lower bound for $\norm{\mathbf{s}-\mathbf{p}}$. Let $L_{\mathbf{s}}= \norm{\mathbf{s}-\mathbf{p}} + \norm{\mathbf{s}-\mathbf{a}} + \norm{\mathbf{s}-\mathbf{b}}$. Without loss of generality, suppose that $\norm{\mathbf{p}-\mathbf{a}} \geq \norm{\mathbf{p}-\mathbf{b}}$. We will now show that $\norm{\mathbf{a}-\mathbf{b}} \leq \norm{\mathbf{p}-\mathbf{b}}$. It holds that

    \begin{align*}
        \norm{\mathbf{a}-\mathbf{b}}^{2} &= \sum_{i=2}^{n} (a_{i}-b_{i})^{2}\\
                  &= \sum_{i=2}^{n} (a_{i}^{2} - 2a_{i}b_{i} + b_{i}^{2})\\
                  &= \norm{\mathbf{p}-\mathbf{b}}^{2} - 1 + \sum_{i=2}^{n} (a_{i}^{2} - 2a_{i}b_{i})\\
                  &\leq \norm{\mathbf{p}-\mathbf{b}}^{2},
    \end{align*}
    since $\sum_{i=2}^{n} a_{i}^{2} = \left(\sum_{i=2}^{n} a_{i}\right)^{2} - \sum_{i \neq j} 2a_{i}a_{j} = 1 - \sum_{i \neq j} 2a_{i}a_{j} \leq 1$. The third step uses that $\mathbf{p} = \mathbf{e_1}$ and $\mathbf{b} \in F_1$, so $\mathbf{p}$ and $\mathbf{b}$ are orthogonal.
    
    From the proof of the Gilbert-Pollak Steiner ratio conjecture for 3 points (\S10 of \cite{gilpol}), it is known that for $L=\norm{\mathbf{a}-\mathbf{b}}+\norm{\mathbf{p}-\mathbf{b}}$ it holds that
    \[
    L \geq L_\mathbf{s} \geq \frac{\sqrt{3}}{2} L,
    \]
    by considering $\mathbf{p}$, $\mathbf{a}$, $\mathbf{b}$ as a set of terminals for the Steiner Tree problem. Next, it is important to note that $\norm{\mathbf{p}-\mathbf{b}} > 1$ and $\mathbf{b} \in F_1$. Finally, we can use formula (18) of \cite{gilchung} to compute a lower bound for $\norm{\mathbf{s}-\mathbf{p}}$:

    \begin{align*}
        \norm{\mathbf{s}-\mathbf{p}} &= \frac{L_\mathbf{s}+\frac{\norm{\mathbf{p}-\mathbf{b}}^{2}+\norm{\mathbf{p}-\mathbf{a}}^{2}-2\norm{\mathbf{a}-\mathbf{b}}^{2}}{L_{\mathbf{s}}}}{3}\\
              &\geq \frac{\frac{\sqrt{3}}{2}L^{2} + \norm{\mathbf{p}-\mathbf{b}}^{2}+\norm{\mathbf{p}-\mathbf{a}}^{2}-2\norm{\mathbf{a}-\mathbf{b}}^{2}}{3L}.
    \end{align*}
    Then, 

    \begin{align*}
        \norm{\mathbf{s}-\mathbf{p}} &\geq \frac{\norm{\mathbf{a}-\mathbf{b}}^{2}\left(\frac{\sqrt{3}}{2}-2\right)+\norm{\mathbf{p}-\mathbf{b}}^{2}\left(\frac{\sqrt{3}}{2}+1\right)+\norm{\mathbf{p}-\mathbf{a}}^{2}+\sqrt{3}\norm{\mathbf{a}-\mathbf{b}}\norm{\mathbf{p}-\mathbf{b}}}{3\left(\norm{\mathbf{a}-\mathbf{b}}+\norm{\mathbf{p}-\mathbf{b}}\right)}\\
                   &= \frac{\norm{\mathbf{p}-\mathbf{a}}^{2}-\norm{\mathbf{a}-\mathbf{b}}^{2}+\left(\norm{\mathbf{a}-\mathbf{b}}+\norm{\mathbf{p}-\mathbf{b}}\right)\left(\norm{\mathbf{a}-\mathbf{b}}\left(\frac{\sqrt{3}}{2}-1\right)+\norm{\mathbf{p}-\mathbf{b}}\left(\frac{\sqrt{3}}{2}+1\right)\right)}{3\left(\norm{\mathbf{a}-\mathbf{b}}+\norm{\mathbf{p}-\mathbf{b}}\right)}\\
                   &= \underbrace{\frac{\norm{\mathbf{p}-\mathbf{a}}^{2}-\norm{\mathbf{a}-\mathbf{b}}^{2}}{3\left(\norm{\mathbf{a}-\mathbf{b}}+\norm{\mathbf{p}-\mathbf{b}}\right)}}_{\geq 0}+ \frac{1}{3} \cdot \left( \norm{\mathbf{a}-\mathbf{b}}\left(\frac{\sqrt{3}}{2}-1\right)+\norm{\mathbf{p}-\mathbf{b}}\left(\frac{\sqrt{3}}{2}+1\right) \right).
    \end{align*}
    
    Finally, since $\frac{\sqrt{3}}{2}-1 < 0$ and $\norm{\mathbf{a}-\mathbf{b}} \leq \norm{\mathbf{p}-\mathbf{b}}$, it holds that

    \begin{align*}
        \norm{\mathbf{a}-\mathbf{b}}\left(\frac{\sqrt{3}}{2}-1\right)+\norm{\mathbf{p}-\mathbf{b}}\left(\frac{\sqrt{3}}{2}+1\right) &\geq \norm{\mathbf{p}-\mathbf{b}}\left(\frac{\sqrt{3}}{2}-1+\frac{\sqrt{3}}{2}+1\right)\\
                              &> \sqrt{3}.
    \end{align*}
    Therefore, $\norm{\mathbf{s}-\mathbf{p}} > \frac{1}{\sqrt{3}}$.
\end{proof}

\begin{remark}
 From Lemma \ref{symmetries} we already know that the lengths for a pair of terminals adjacent to the same Steiner point are the same and can hence compute the lengths using Theorem \ref{onetwenty}. In contrast, we did not previously know how to lower bound the length of an edge between a Steiner point and a terminal when the Steiner point is adjacent to only one terminal. Theorem \ref{orphanlength} now provides such a bound.
\end{remark}
\subsection{Constructing Steiner trees for the regular simplex}
To describe our construction of Steiner trees of the regular simplex, we need one more definition.
\begin{definition}[The Split of a Point]
    For $\mathbf{x} \in \R^d$, we define the \textit{split} of $\mathbf{x}$ to be $\mathbf{x'} \in \R^{2d}$ such that $\mathbf{x'}=\left(\frac{x_{1}}{2},\frac{x_{1}}{2},\ldots,\frac{x_{d}}{2},\frac{x_{d}}{2}\right)$. 
\end{definition}

\begin{lemma}
\label{spl}
    Let $\mathbf{x},\mathbf{y},\mathbf{z} \in \mathbb{R}^{d}$ such that the angle included by $\mathbf{x}-\mathbf{y}$ and $\mathbf{z}-\mathbf{y}$ is $\alpha$. Let $\mathbf{x'}, \mathbf{y'}, $ and $\mathbf{z'}$ denote the splits of $\mathbf{x}, \mathbf{y}, $ and $\mathbf{z}$, respectively. Then, the angle included by $\mathbf{x'}-\mathbf{y'}$ and $\mathbf{z'}-\mathbf{y'}$ is $\alpha$.
\end{lemma}

\begin{proof}
    Let $\mathbf{u}=(\mathbf{x}-\mathbf{y})$, $\mathbf{v}=(\mathbf{z}-\mathbf{y})$, $\mathbf{u'}=(\mathbf{x'}-\mathbf{y'})$ and $\mathbf{v'}=(\mathbf{z'}-\mathbf{y'})$. Now $\mathbf{u} \cdot \mathbf{v}$ = $2\mathbf{u'} \cdot \mathbf{v'}$, $\norm{\mathbf{u} }=\sqrt{2}\norm{\mathbf{u'}}$ and the same with the norms of $\mathbf{v}$ and $\mathbf{v'}$. Now it is easy to see $\frac{\mathbf{u} \cdot \mathbf{v}}{\norm{\mathbf{u}} \norm{\mathbf{v}}} = \frac{\mathbf{u'} \cdot \mathbf{v'}}{\norm{\mathbf{u'}} \norm{\mathbf{v'}}}$.
\end{proof}

We will show how to leverage this idea of split to explicitly construct candidate-optimal Steiner trees of the regular simplex from optimal Steiner trees of smaller regular simplices (which can be computed directly). We formalize this in the following definition.
\begin{definition}[Candidate-optimal Steiner tree]
A Steiner tree $T$ of a point configuration is called a \textit{candidate-optimal Steiner tree} if each pair of edges incident to a common vertex in $T$ include an angle of at least 120 degrees and each Steiner point is of degree exactly $3$. 
\end{definition}

Finally, we describe how to construct a candidate-optimal, full Steiner tree of the regular $2d$-simplex from a candidate-optimal full Steiner tree of the regular $d$-simplex. In the below, we will use the fact that all terminals are leaf nodes in an optimal Steiner tree if and only if it is full. This is clear by computing the degree sum of the tree in two ways: using that Steiner points are all degree $3$ and that a full Steiner tree on $n$ terminals is a tree on $2n-2$ total nodes.
\begin{theorem}[Doubling the Tree]
\label{twice}
    Let $d \geq 3$ and let $T$ be an optimal Steiner tree of a regular $d$-simplex. Let $S$ be the set of Steiner points in $T$. Then by the following procedure, we obtain a full, candidate-optimal Steiner tree $T'$ of the regular $2d$-simplex:
    \begin{enumerate}
        \item For each $\mathbf{s} \in S$, denote the split of $\mathbf{s}$ by $\mathbf{s'}$. The set $\{\mathbf{s'} \,:\, \mathbf{s} \in S\}$ is a subset of the Steiner points $S'$ in $T'$.
        \item For all $\mathbf{r},\mathbf{s} \in S$ such that $(\mathbf{r}, \,\mathbf{s}) \in E(T)$, let $(\mathbf{r'}, \,\mathbf{s'}) \in E(T')$.
        \item For each terminal $\mathbf{e_i} \in V(T)$, let $\mathbf{s_i}$ be the adjacent Steiner point. Then we obtain new Steiner points $\mathbf{x_i}$ in $T'$ by finding the Fermat points of the triangles formed by $\mathbf{e_{2i-1}}, \mathbf{e_{2i}}, \mathbf{s'_i}$.
        \item Add edges $(\mathbf{e_{2i-1}}, \,\mathbf{x_i})$, $(\mathbf{e_{2i}}, \,\mathbf{x_i})$ and $(\mathbf{x_i}, \,\mathbf{s'_i})$ to $T'$.
    \end{enumerate}
\end{theorem}

\begin{proof}
Note that, assuming that the Fermat points of the triangles of the form $\triangle \mathbf{e_{2i-1}} \mathbf{e_{2i}} \mathbf{s'_i}$ exist, $T'$ is a full Steiner tree of the regular $2d$-simplex. The set of Steiner points $\{\mathbf{s'} \,:\, \mathbf{s} \in T\}$ are connected via the same tree topology as the Steiner points in $T$. Then, for each $i \in [d]$, each terminal $\mathbf{e_{2i}}$ or $\mathbf{e_{2i-1}}$ is connected to this tree via the Steiner tree of the triangle $\triangle \mathbf{e_{2i-1}} \mathbf{e_{2i}} \mathbf{s'_i}$ (via the additional Steiner point $\mathbf{x_i}$). Since one Steiner point is added to $T'$ for each Steiner point and terminal in $T$ and $T$ is a full Steiner tree of the regular $d$-simplex, $T'$ has $2d - 2$ total Steiner points and is also full (this uses Corollary \ref{leaves} and the discussion preceding this theorem).

We need to verify two claims. First, we need to show that the Fermat points always exist for the third step of the construction. Secondly, we need to prove that every included angle between two adjacent edges is $120\degree$ (so that $T'$ is a candidate-optimal Steiner tree). First, consider the terminal $\mathbf{e_i}$ in $T$. 

To prove the first claim, we distinguish between two cases. 
First, suppose that $\mathbf{e_i}$ shares its neighboring Steiner point with another terminal in $T$. In our coordinate system, the distance between every two terminals is exactly $\sqrt{2}$. By Theorem \ref{unique}, it holds
    \begin{equation*}
        \norm{\mathbf{e_i}-\mathbf{s_i}}=\norm{\mathbf{e_j}-\mathbf{s_i}}=\sqrt{\frac{2}{3}},
    \end{equation*}
    applying Lemma \ref{symmetries} and Theorem \ref{onetwenty}.

    Now, we denote the center of the segment $\mathbf{e_{2i-1}}\mathbf{e_{2i}}$ by $\mathbf{c_i}$. It holds that $\mathbf{c_i}$ is the split of $\mathbf{e_i}$. Therefore (following from the proof of Lemma \ref{spl}),
    \begin{equation*}
        \norm{\mathbf{c_i}-\mathbf{s'_{i}}}=\frac{1}{\sqrt{2}}\norm{\mathbf{e_i}-\mathbf{s_i}}=\frac{1}{\sqrt{3}}.
    \end{equation*}
    To get the Fermat point of $\triangle \mathbf{e_{2i-1}}\mathbf{e_{2i}}\mathbf{s'_{i}}$, basic trigonometry tells us that we need to find a point at distance $\frac{1}{\sqrt{6}}$ from $\mathbf{c_i}$ in the direction of $\mathbf{s'_{i}}$ (see Figure \ref{fig: trigonometry}). Since $\norm{\mathbf{c_i}-\mathbf{s'_{i}}} > \frac{1}{\sqrt{6}}$, the Fermat point $\mathbf{x_i}$ does exist.


\begin{figure}[h]
    \centering

\tikzset{every picture/.style={line width=0.75pt}} 

\begin{tikzpicture}[x=0.75pt,y=0.75pt,yscale=-1,xscale=1]

\draw   (323.38,31.85) -- (461,149.55) -- (185.75,149.55) -- cycle ;
\draw  [dash pattern={on 0.84pt off 2.51pt}]  (323.55,31.85) -- (323.91,149.55) ;
\draw   (323.38,89.15) -- (461,149.55) -- (185.75,149.55) -- cycle ;
\draw  [fill={rgb, 255:red, 0; green, 0; blue, 0 }  ,fill opacity=1 ] (320.79,34) .. controls (321.98,35.53) and (324.18,35.81) .. (325.71,34.62) .. controls (327.24,33.43) and (327.51,31.22) .. (326.32,29.69) .. controls (325.13,28.16) and (322.93,27.89) .. (321.4,29.08) .. controls (319.87,30.27) and (319.6,32.47) .. (320.79,34) -- cycle ;
\draw  [fill={rgb, 255:red, 0; green, 0; blue, 0 }  ,fill opacity=1 ] (320.61,91.3) .. controls (321.8,92.83) and (324,93.11) .. (325.53,91.92) .. controls (327.06,90.73) and (327.33,88.52) .. (326.14,86.99) .. controls (324.95,85.46) and (322.75,85.19) .. (321.22,86.37) .. controls (319.69,87.56) and (319.42,89.77) .. (320.61,91.3) -- cycle ;
\draw  [fill={rgb, 255:red, 0; green, 0; blue, 0 }  ,fill opacity=1 ] (321.15,151.71) .. controls (322.34,153.24) and (324.54,153.51) .. (326.07,152.32) .. controls (327.6,151.13) and (327.87,148.93) .. (326.68,147.4) .. controls (325.49,145.87) and (323.29,145.59) .. (321.76,146.78) .. controls (320.23,147.97) and (319.96,150.18) .. (321.15,151.71) -- cycle ;
\draw  [fill={rgb, 255:red, 208; green, 2; blue, 27 }  ,fill opacity=1 ] (463.77,147.4) .. controls (462.58,145.87) and (460.37,145.59) .. (458.84,146.78) .. controls (457.31,147.97) and (457.04,150.18) .. (458.23,151.71) .. controls (459.42,153.24) and (461.63,153.51) .. (463.16,152.32) .. controls (464.69,151.13) and (464.96,148.93) .. (463.77,147.4) -- cycle ;
\draw  [fill={rgb, 255:red, 208; green, 2; blue, 27 }  ,fill opacity=1 ] (188.52,147.4) .. controls (187.33,145.87) and (185.12,145.59) .. (183.59,146.78) .. controls (182.06,147.97) and (181.79,150.18) .. (182.98,151.71) .. controls (184.17,153.24) and (186.38,153.51) .. (187.91,152.32) .. controls (189.44,151.13) and (189.71,148.93) .. (188.52,147.4) -- cycle ;
\draw    (334,104.1) -- (330,98.1) ;
\draw    (313,104.1) -- (317,99.1) ;
\draw   (336.59,95.03) .. controls (334.48,100.21) and (329.64,103.9) .. (323.96,104.14) -- (323.38,89.15) -- cycle ;
\draw   (310.16,95.03) .. controls (312.35,100.39) and (317.44,104.15) .. (323.38,104.15) .. controls (323.55,104.15) and (323.72,104.15) .. (323.9,104.14) -- (323.38,89.15) -- cycle ;

\draw (321.55,28.45) node [anchor=south east] [inner sep=0.75pt]    {$\mathbf{s} '_{\mathbf{i}}$};
\draw (321.38,87.3) node [anchor=south east] [inner sep=0.75pt]    {$\mathbf{x}_{\mathbf{i}}$};
\draw (321.91,152.95) node [anchor=north east] [inner sep=0.75pt]    {$\mathbf{c}_{\mathbf{i}}$};
\draw (343.22,104.77) node [anchor=north] [inner sep=0.75pt]  [color={rgb, 255:red, 0; green, 0; blue, 0 }  ,opacity=1 ]  {$6\mathnormal{0\degree }$};
\draw (377,172.4) node [anchor=north west][inner sep=0.75pt]  [font=\normalsize]  {$\frac{1}{\sqrt{2}}$};
\draw (334.67,177.69) node [anchor=north west][inner sep=0.75pt]  [font=\Huge,rotate=-272,xslant=-0.04]  {$\Biggl\{$};
\draw (185.75,152.95) node [anchor=north] [inner sep=0.75pt]    {$\mathbf{e_{2i-1}}$};
\draw (463,152.95) node [anchor=north west][inner sep=0.75pt]    {$\mathbf{e_{2i}}$};

\end{tikzpicture}

    \caption{Diagram showing the existence of the Fermat point $\mathbf{x_i}$. Note that $\norm{\mathbf{c_{i}}-\mathbf{s'_{i}}} > \frac{1}{\sqrt{2}}\cot(60^{\circ})=\frac{1}{\sqrt{6}}$.}
    \label{fig: trigonometry}
\end{figure}
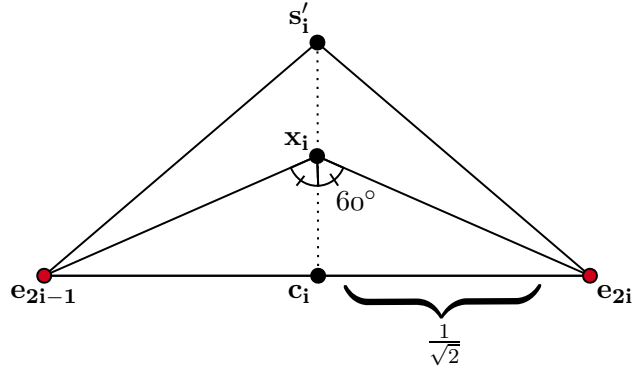

In the case that $\mathbf{e_i}$ does not share its neighboring Steiner point with another terminal in $T$, $\mathbf{c_i}$ is still well-defined as above. Using $\norm{\mathbf{c_i}-\mathbf{s'_{i}}}=\frac{1}{\sqrt{2}}\norm{\mathbf{e_i}-\mathbf{s_i}}$ and that $\norm{\mathbf{e_{2i-1}} - \mathbf{s_i'}} = \norm{\mathbf{e_{2i}} - \mathbf{s_i'}}$ by definition of split, we need to find a point at distance $\frac{1}{\sqrt{6}}$ from $\mathbf{c_i}$ in the direction of $\mathbf{s'_{i}}$ to find a Fermat point of $\triangle \mathbf{e_{2i-1}}\mathbf{e_{2i}}\mathbf{s'_{i}}$. Hence, a Fermat point of this triangle exists if and only if $\norm{\mathbf{e_i}-\mathbf{s_i}} > 1/\sqrt{3}$ and a Fermat point then exists by Theorem \ref{orphanlength}.

 To prove the second claim, observe that all such the angles were equal to $120\degree$ in $T$ (by Theorem \ref{onetwenty}). Lemma \ref{spl} then implies that the angles stay the same size in the induced subtree on Steiner points in $\{\mathbf{s'} \,:\, \mathbf{s} \in T\}$. We already established that $\mathbf{c_i}$ is the split of $\mathbf{e_i}$, so, $\mathbf{e_i}$ and $\mathbf{e_j}$ share an adjacent Steiner point in $T$ and $\mathbf{c_j}$ is the center of the segment $\mathbf{e_{2j-1}}\mathbf{e_{2j}}$, then 
$|\angle \mathbf{c_i}\mathbf{s'_{i}}\mathbf{c_j}|$ in $T'$ is equal to $|\angle \mathbf{e_i}\mathbf{s_i}\mathbf{p_j}|$ in $T$ which is $120\degree$ by Theorem \ref{onetwenty}. Since the points $\mathbf{c_i}, \mathbf{x_i}, \mathbf{s'_{i}}$ and $\mathbf{c_j}, \mathbf{x_j}, \mathbf{s'_{i}}$ are colinear from the second part of Lemma \ref{symmetries} we have that  \begin{equation*}
        |\angle \mathbf{x_i}\mathbf{s'_{i}}\mathbf{x_j}|=120\degree.
    \end{equation*} 

Finally, the remaining angles involve the Steiner points added which were computed as Fermat points of triangles. By definition of Fermat points, the edges sharing an endpoint at these Steiner points include angles of exactly $120\degree$.
\end{proof}

\begin{corollary} \label{cor: twice rep}
Let $d \geq 3, k \geq 0$ and let $T$ be an optimal Steiner tree of a regular $d$-simplex. Repeating the procedure in Theorem \ref{twice},  $k$ many times, yields a full, candidate-optimal Steiner tree $T'$ of a regular $2^kd$-simplex.
\end{corollary}
\begin{proof}
Optimality of $T$ handles $k = 0$ (using that the optimal Steiner tree of the regular simplex is full, using the discussion preceding Theorem \ref{twice} and Lemma \ref{lem: max pt prop}). Theorem \ref{twice} handles the case of $k = 1$. Assume the result holds for up to some $r \geq 1$ and let $T$ be the optimal tree with Steiner points $S$. We show the result for $ k = r +1$. The same proof as in Theorem \ref{twice} shows that $T'$ computed from $T$ by the procedure in Theorem \ref{twice} is a full Steiner tree of the regular simplex, assuming that the relevant Fermat points exist.

It remains to show that the Fermat points exist in each subsequent iteration and all the angles formed by edges at a common endpoint are $120\degree$. 

First, note that, since we assumed $r \geq 1$, by the procedure in Theorem \ref{twice}, each terminal in $T$ shares an adjacent Steiner point with another terminal. Namely, $\mathbf{e_{2i}}$ shares an adjacent Steiner point $\mathbf{s_{2i}}$ with $\mathbf{e_{2i-1}}$ for each $i \in [2^{r-1}d]$. The third neighbor $\mathbf{r}$ of $\mathbf{s_{2i}}$ in $T$ is the split of the Steiner point neighboring $\mathbf{e_i}$ in the tree preceding $T$. In particular, then $\norm{\mathbf{r} - \mathbf{e_{2i}}} = \norm{\mathbf{r} - \mathbf{e_{2i - 1}}}$, so $\triangle \mathbf{r} \mathbf{e_{2i - 1}} \mathbf{e_{2i}}$ is isosceles. Hence, $\mathbf{s_{2i}}$, the Fermat point of this triangle must be equidistant to $\mathbf{e_{2i - 1}}$ and $\mathbf{e_{2i}}$. Then,  
\[
\norm{\mathbf{e_{2i - 1}} - \mathbf{s_{2i}}} = \norm{\mathbf{e_{2i}} - \mathbf{s_{2i}}} = \sqrt{\frac{2}{3}}
\]
using that the angles includes by the edges to $\mathbf{s_{2i}}$ must be $120\degree$ (since it is a Fermat point) and the distance between terminals is $\sqrt{2}$. Now, as in Theorem \ref{twice}, for the Fermat point of $\triangle{\mathbf{e_{4i}} \mathbf{e_{4i-1}} \mathbf{s_{2i}'}}$ to exist in $T'$, we need that the center of $\mathbf{e_{4i}} \mathbf{e_{4i-1}}$, $\mathbf{c_{2i}}$, is at least $\frac{1}{\sqrt{6}}$ from $\mathbf{s_{2i}'}$. But, $\mathbf{c_{2i}}$ is the split of $\mathbf{e_{2i}}$ and $\norm{\mathbf{e_{2i}} - \mathbf{s_{2i}}}$ is $\sqrt{\frac{2}{3}}$ from the above. Hence, by Lemma \ref{spl}, we have $\norm{\mathbf{c_{2i}} - \mathbf{s_{2i}'}} = \frac{1}{\sqrt{3}} > \frac{1}{\sqrt{6}}$, as necessary. So, the Fermat points exist in constructing $T'$.

Finally, we need to show that all included angles between edges sharing an endpoint are at least $120\degree$ in $T'$.
The argument here is identical to the argument in Theorem \ref{twice} (after applying the inductive hypothesis), completing the proof.
\end{proof}

Consider the Steiner trees of $d$-dimensional simplices for some small value of $d$ where we can determine explicit coordinates for every Steiner point (e.g., $d=3,4$). The construction described in Theorem \ref{twice} yields the same Steiner points as the numerical algorithm in \cite{smith} for $d=6,8,12$ (up to small errors presumably caused by the approximate nature of Smith's algorithm). For higher values of $d$ it was not checked due to computational limitations.

\subsection{Explicit construction for $d=2^{k}$}

Applying Corollary \ref{cor: twice rep} starting from $d = 4$ yields an explicit construction for Steiner trees of regular simplices on $n=2^{k}$ terminals. We analyze that construction in detail in this section. To start, find explicit coordinates for the Steiner points of a Steiner tree of the simplex on $d=4$ terminals. Then, apply Theorem \ref{twice}, $(k-2)$ many times.

Our topology will be given by two full binary trees $T^0,T^1$, each on $2^{k-1}$ terminals, and an edge connecting both roots. Now recall Definition \ref{ATB} and label $T^0$ with respect to $0$ and $T^1$ with respect to $1$. This is the representation of our tree that we will work with (see Figure \ref{fig: binary representation of points}). For simplicity, when we will talk about coordinates related to a terminal, we will use their unique binary label instead. The labels $i$ will be in the range from $0$ to $d-1$ (in binary). Therefore, the terminal with label $i$ will be $\mathbf{e_{i+1}}$.

For $k \in \mathbb{N}$, let $\{T_{m}\}^{k}_{m=2}$ denote the sequence of trees from our construction. Each vertex has a superscript and a subscript---the superscript $m$ refers to the vertex belonging to the vertex set of $T_m$ and the subscript refers to the assigned binary string label. Steiner points obtained by splitting will retain the same binary string label. Those obtained as a new Fermat point will adopt the binary string label of the terminal whose binary label was appended to form the two terminal endpoints of the triangle. E.g., the Steiner point obtained as the split of $\mathbf{s^k_b}$ will be $\mathbf{s^{k+1}_b}$. The Steiner point obtained as the Fermat point of $\triangle \mathbf{s^{k+1}_b}\mathbf{p^{k+1}_{b00}}\mathbf{p^{k+1}_{b01}}$ will be $\mathbf{s^{k+1}_{b0}}$.

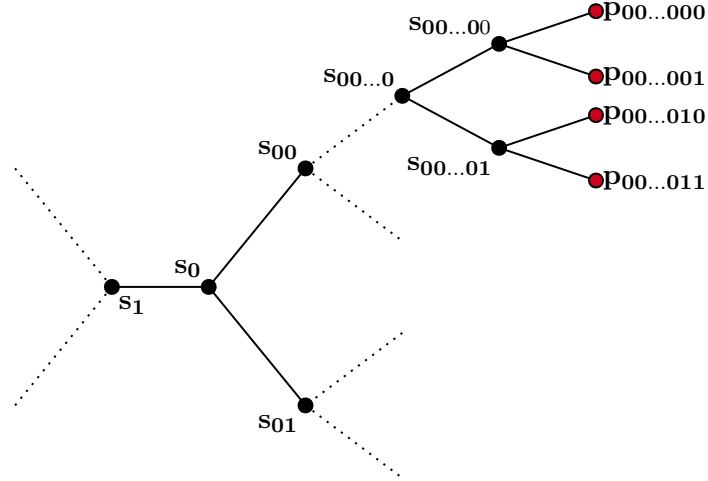
\begin{figure}[h]
    \centering

\tikzset{every picture/.style={line width=0.75pt}} 

\begin{tikzpicture}[x=0.75pt,y=0.75pt,yscale=-1,xscale=1]

\draw    (161.83,162.18) -- (210.66,162.18) ;
\draw  [dash pattern={on 0.84pt off 2.51pt}]  (113,102.37) -- (161.83,162.18) ;
\draw  [dash pattern={on 0.84pt off 2.51pt}]  (113,222) -- (161.83,162.18) ;
\draw    (259.49,102.37) -- (210.66,162.18) ;
\draw    (210.66,162.18) -- (259.49,222) ;
\draw  [dash pattern={on 0.84pt off 2.51pt}]  (308.32,185.38) -- (259.49,222) ;
\draw  [dash pattern={on 0.84pt off 2.51pt}]  (259.49,222) -- (308.32,258.62) ;
\draw  [dash pattern={on 0.84pt off 2.51pt}]  (259.49,102.37) -- (308.32,138.99) ;
\draw  [dash pattern={on 0.84pt off 2.51pt}]  (259.49,102.37) -- (308.32,65.75) ;
\draw    (308.32,65.75) -- (357.15,91.99) ;
\draw    (308.32,65.75) -- (357.15,39.5) ;
\draw    (357.15,39.5) -- (405.98,23) ;
\draw    (357.15,39.5) -- (405.98,56) ;
\draw    (357.15,91.99) -- (405.98,108.49) ;
\draw    (357.15,91.99) -- (405.98,75.5) ;
\draw  [fill={rgb, 255:red, 208; green, 2; blue, 27 }  ,fill opacity=1 ] (408.74,20.85) .. controls (407.55,19.32) and (405.35,19.04) .. (403.82,20.23) .. controls (402.29,21.42) and (402.02,23.63) .. (403.21,25.16) .. controls (404.4,26.69) and (406.61,26.97) .. (408.13,25.78) .. controls (409.66,24.59) and (409.94,22.38) .. (408.74,20.85) -- cycle ;
\draw  [fill={rgb, 255:red, 208; green, 2; blue, 27 }  ,fill opacity=1 ] (408.74,53.84) .. controls (407.55,52.31) and (405.35,52.03) .. (403.82,53.22) .. controls (402.29,54.41) and (402.02,56.62) .. (403.21,58.15) .. controls (404.4,59.68) and (406.61,59.96) .. (408.13,58.77) .. controls (409.66,57.58) and (409.94,55.37) .. (408.74,53.84) -- cycle ;
\draw  [fill={rgb, 255:red, 208; green, 2; blue, 27 }  ,fill opacity=1 ] (408.74,73.34) .. controls (407.55,71.81) and (405.35,71.53) .. (403.82,72.72) .. controls (402.29,73.91) and (402.02,76.12) .. (403.21,77.65) .. controls (404.4,79.18) and (406.61,79.46) .. (408.13,78.27) .. controls (409.66,77.08) and (409.94,74.87) .. (408.74,73.34) -- cycle ;
\draw  [fill={rgb, 255:red, 208; green, 2; blue, 27 }  ,fill opacity=1 ] (408.74,106.33) .. controls (407.55,104.8) and (405.35,104.53) .. (403.82,105.72) .. controls (402.29,106.91) and (402.02,109.11) .. (403.21,110.64) .. controls (404.4,112.17) and (406.61,112.45) .. (408.13,111.26) .. controls (409.66,110.07) and (409.94,107.86) .. (408.74,106.33) -- cycle ;
\draw  [fill={rgb, 255:red, 0; green, 0; blue, 0 }  ,fill opacity=1 ] (359.91,41.65) .. controls (358.72,43.18) and (356.52,43.46) .. (354.99,42.27) .. controls (353.46,41.08) and (353.19,38.88) .. (354.38,37.35) .. controls (355.57,35.82) and (357.78,35.54) .. (359.3,36.73) .. controls (360.83,37.92) and (361.11,40.12) .. (359.91,41.65) -- cycle ;
\draw  [fill={rgb, 255:red, 0; green, 0; blue, 0 }  ,fill opacity=1 ] (359.91,94.15) .. controls (358.72,95.68) and (356.52,95.95) .. (354.99,94.76) .. controls (353.46,93.57) and (353.19,91.37) .. (354.38,89.84) .. controls (355.57,88.31) and (357.78,88.03) .. (359.3,89.22) .. controls (360.83,90.41) and (361.11,92.61) .. (359.91,94.15) -- cycle ;
\draw  [fill={rgb, 255:red, 0; green, 0; blue, 0 }  ,fill opacity=1 ] (311.09,67.9) .. controls (309.89,69.43) and (307.69,69.71) .. (306.16,68.52) .. controls (304.63,67.33) and (304.36,65.12) .. (305.55,63.59) .. controls (306.74,62.06) and (308.95,61.78) .. (310.47,62.97) .. controls (312,64.16) and (312.28,66.37) .. (311.09,67.9) -- cycle ;
\draw  [fill={rgb, 255:red, 0; green, 0; blue, 0 }  ,fill opacity=1 ] (262.26,104.52) .. controls (261.06,106.05) and (258.86,106.33) .. (257.33,105.14) .. controls (255.8,103.95) and (255.53,101.75) .. (256.72,100.21) .. controls (257.91,98.68) and (260.12,98.41) .. (261.65,99.6) .. controls (263.17,100.79) and (263.45,102.99) .. (262.26,104.52) -- cycle ;
\draw  [fill={rgb, 255:red, 0; green, 0; blue, 0 }  ,fill opacity=1 ] (213.43,164.34) .. controls (212.23,165.87) and (210.03,166.15) .. (208.5,164.96) .. controls (206.97,163.77) and (206.7,161.56) .. (207.89,160.03) .. controls (209.08,158.5) and (211.29,158.22) .. (212.82,159.41) .. controls (214.34,160.6) and (214.62,162.81) .. (213.43,164.34) -- cycle ;
\draw  [fill={rgb, 255:red, 0; green, 0; blue, 0 }  ,fill opacity=1 ] (262.26,224.15) .. controls (261.06,225.68) and (258.86,225.96) .. (257.33,224.77) .. controls (255.8,223.58) and (255.53,221.38) .. (256.72,219.85) .. controls (257.91,218.32) and (260.12,218.04) .. (261.65,219.23) .. controls (263.17,220.42) and (263.45,222.62) .. (262.26,224.15) -- cycle ;
\draw  [fill={rgb, 255:red, 0; green, 0; blue, 0 }  ,fill opacity=1 ] (164.6,164.34) .. controls (163.41,165.87) and (161.2,166.15) .. (159.67,164.96) .. controls (158.14,163.77) and (157.87,161.56) .. (159.06,160.03) .. controls (160.25,158.5) and (162.46,158.22) .. (163.99,159.41) .. controls (165.51,160.6) and (165.79,162.81) .. (164.6,164.34) -- cycle ;

\draw (407.98,23) node [anchor=west] [inner sep=0.75pt]    {$\mathbf{p}_{\mathbf{00...000}}$};
\draw (407.98,56) node [anchor=west] [inner sep=0.75pt]    {$\mathbf{p}_{\mathbf{00...001}}$};
\draw (407.98,75.5) node [anchor=west] [inner sep=0.75pt]    {$\mathbf{p}_{\mathbf{00...010}}$};
\draw (407.98,108.49) node [anchor=west] [inner sep=0.75pt]    {$\mathbf{p}_{\mathbf{00...011}}$};
\draw (208.66,158.78) node [anchor=south east] [inner sep=0.75pt]    {$\mathbf{s}_{\mathbf{0}}$};
\draw (163.83,165.58) node [anchor=north west][inner sep=0.75pt]    {$\mathbf{s}_{\mathbf{1}}$};
\draw (257.49,98.97) node [anchor=south east] [inner sep=0.75pt]    {$\mathbf{s}_{\mathbf{00}}$};
\draw (257.49,225.4) node [anchor=north east] [inner sep=0.75pt]    {$\mathbf{s}_{\mathbf{01}}$};
\draw (306.32,62.35) node [anchor=south east] [inner sep=0.75pt]    {$\mathbf{s}_{\mathbf{00...0}}$};
\draw (355.15,95.39) node [anchor=north east] [inner sep=0.75pt]    {$\mathbf{s}_{\mathbf{00...01}}$};
\draw (355.15,36.1) node [anchor=south east] [inner sep=0.75pt]    {$\mathbf{s}_{\mathbf{00...0} 0}$};

\end{tikzpicture}

    \caption{Binary representation of terminals and Steiner points for $n=2^{k}$.}
    \label{fig: binary representation of points}
\end{figure}

For $k=1$, $T_{1}$ is just the line segment $\mathbf{p^1_{0}}\mathbf{p^1_{1}}$ connecting two terminals. For $k=2$, we need to find the Steiner points $\mathbf{s^{2}_{0}}$ and $\mathbf{s^{2}_{1}}$. Lemma \ref{symmetries} gives us the following: the edge between $\mathbf{s^{2}_{0}}$ and $\mathbf{s^{2}_{1}}$ passes through  $\mathbf{c} = \left( \frac{1}{4},\frac{1}{4},\frac{1}{4},\frac{1}{4} \right)$, the centroid of the terminals. Then, it holds that $\mathbf{s^{2}_{0}}$ and $\mathbf{s^{2}_{1}}$ are the Fermat points of $\triangle \mathbf{p^2_{00}}\mathbf{p^2_{01}}\mathbf{c}$ and $\triangle \mathbf{p^2_{10}}\mathbf{p^2_{11}}\mathbf{c}$, respectively. Therefore, we get:
\begin{align*}
    \mathbf{s^{2}_{0}}&=\left( \frac{1}{2} - \frac{1}{2\sqrt{6}},\frac{1}{2} - \frac{1}{2\sqrt{6}},\frac{1}{2\sqrt{6}},\frac{1}{2\sqrt{6}} \right),\\
    \mathbf{s^{2}_{1}}&=\left( \frac{1}{2\sqrt{6}},\frac{1}{2\sqrt{6}},\frac{1}{2} - \frac{1}{2\sqrt{6}},\frac{1}{2} - \frac{1}{2\sqrt{6}} \right).
\end{align*}
For $k \geq 3$, let $\{b_{j}\}^{k-1}_{j=1}$ be the sequence of binary strings of $j$ zeros. For brevity, we will only show the coordinates of points $\mathbf{s^{k}_{b_{j}}}$. To obtain explicit formulas for other Steiner points, it suffices to apply suitable topology-preserving coordinate permutations.

Firstly, based on Theorem \ref{twice}, we know that we obtain $\mathbf{s^{k}_{0}}$ by splitting $\mathbf{s^{k-1}_{0}}$. Repeatedly applying this, we have
\begin{equation*}
    \mathbf{s^{k}_{0}} = \frac{1}{2^{k-2}}\vvector{\frac{1}{2}-\frac{1}{2\sqrt{6}}}{2^{k-1} \, \text{times}}{\frac{1}{2\sqrt{6}}}.
\end{equation*}
Now for each $2 \leq j \leq k-1$, let us find the first tree in our sequence that has a Steiner point with the binary representation of $b_{j}$. It is $T_{j+1}$. In this tree, the point $\mathbf{s^{j+1}_{b_{j}}}$ was constructed as the Fermat point of $\triangle \mathbf{p^{j+1}_{b_{j}0}}\mathbf{p^{j+1}_{b_{j}1}}\mathbf{s^{j+1}_{b_{j-1}}}$. If we denote the center of $\mathbf{p^{j+1}_{b_{j}0}}\mathbf{p^{j+1}_{b_{j}1}}$ as $\mathbf{c}^{j+1}$, we already know from the proof of Corollary \ref{cor: twice rep} that 
\begin{equation*}
    \norm{\mathbf{c}^{j+1}\mathbf{s^{j+1}_{b_{j}}}}=\frac{1}{\sqrt{2}}\norm{\mathbf{c}^{j+1}\mathbf{s^{j+1}_{b_{j-1}}}}.
\end{equation*}
From this it follows that
\begin{align*}
    \mathbf{s^{j+1}_{b_{j}}}&=\mathbf{c}^{j+1} + \frac{1}{\sqrt{2}} \left(\mathbf{s^{j+1}_{b_{j-1}}}-    \mathbf{c}^{j+1} \right)\\
                   &=\frac{1}{\sqrt{2}}\mathbf{s^{j+1}_{b_{j-1}}} + \left(1-\frac{1}{\sqrt{2}} \right) \left( \frac{1}{2},\frac{1}{2},0, \cdots, 0 \right)\\
                   &= \frac{1}{\sqrt{2}}\mathbf{s^{j+1}_{b_{j-1}}} + \left( \frac{1}{2} - \frac{1}{2\sqrt{2}},\frac{1}{2} - \frac{1}{2\sqrt{2}},0, \cdots, 0 \right).
\end{align*}

After that, we need to split $\mathbf{s^{j+1}_{b_{j}}}$ a total of $(k-j-1)$ times to obtain $\mathbf{s^{k}_{b_{j}}}$. Splitting is linear and can be done separately on both summands

\begin{equation*}
    \mathbf{s^{k}_{b_{j}}}=\frac{1}{\sqrt{2}}\mathbf{s^{k}_{b_{j-1}}}+\frac{1}{2^{k-j-1}} \vvector{\frac{1}{2}-\frac{1}{2\sqrt{2}}}{2^{k-j} \, \text{times}}{0}.
\end{equation*}

We can then repeat this step with $\mathbf{s^{k}_{b_{j-1}}}$ and so on until we get to $\mathbf{s^{k}_{0}}$

\begin{align*}
    \mathbf{s^{k}_{b_{j}}}= \left( \frac{1}{\sqrt{2}} \right)^{j-1}\mathbf{s^{k}_{0}} 
    &+ \frac{1}{2^{k-j-1}} \vvector{\frac{1}{2} - \frac{1}{2\sqrt{2}}}{2^{k-j} \, \text{times}}{0}\\
    &+ \frac{1}{ \sqrt{2} \cdot 2^{k-j}} \vvector{\frac{1}{2} - \frac{1}{2\sqrt{2}}}{2^{k-j+1} \, \text{times}}{0}\\
    &\vdots\\
    &+\frac{1}{(\sqrt{2})^{j-1}2^{k-2}} \vvector{\frac{1}{2} - \frac{1}{2\sqrt{2}}}{2^{k-1} \, \text{times}}{0}.\\
\end{align*}

Note that this type of construction is not restricted to powers of two: the same can be done for any initial known Steiner tree. For example we can explicitly write down the coordinates for $d=3 \cdot 2^{k}$, $k \in \mathbb{N}$, where the only part of the expression that changes is the point $\mathbf{s^k_{0}}$. Or, by running some exact algorithm, e.g., Smith's algorithm, we can compute numerical approximations for the Steiner points of the optimal Steiner tree for $d = c$ for some some small constant $c$ and then apply the same technique to write down the coordinates for $d = c \cdot 2^k$, $k \in \mathbb{N}$.

\begin{conjecture}
    This construction yields an optimal Steiner tree for every regular $d$-simplex, where $d=2^{k}$, $k \geq 1$. Moreover, the natural generalization yields the optimal Steiner tree for every regular $d$-simplex, $d \geq 3$.
\end{conjecture}

The construction is closely related to the construction in \cite{gilchung}. The outcome in both cases is that, instead of considering every terminal, it is enough to represent each full binary tree by the centroid of its terminals. In essence, this is the second property of optimal Steiner trees of the regular simplex that we formalize in Lemma \ref{symmetries}. It is not surprising to notice that the asymptotic length of our constructions are the same (although our constructions match the conjecture of Smith for all $d$, unlike \cite{gilchung}):

\begin{proposition}
\label{limit}
    Let $T_0$ be a Steiner tree of the regular $d$-simplex. Let $\{T_k\}_{k=0}^{\infty}$ be the sequence of Steiner trees of the regular simplex created by repeatedly applying Theorem \ref{twice} to $T_0$ and let $\ell_k$ denote the Steiner ratio for $T_k$. If $\lim_{k \to \infty} \ell_k$ exists, then $\lim_{k \to \infty} \ell_k=\frac{\sqrt{3}}{\sqrt{2}(2\sqrt{2}-1)}$.
\end{proposition}

\begin{proof}
    Suppose that we know $\ell_0$. Then we can recursively write 
    \begin{equation}
    \label{eq:recursive}
    \ell_{k+1}=\frac{\ell_k\left(d2^k-1\right)-\frac{d2^k}{\sqrt{6}}+d2^{k+1}\sqrt{\frac{2}{3}}}{\left(d2^{k+1}-1\right) \sqrt{2}}. \tag{\textasteriskcentered}
    \end{equation}
    The distance between all pairs of terminals is $\sqrt{2}$, yielding the denominator. The tree resultant from taking the split of every node in $T_k$ (and retaining the same topology) is nearly $T_{k+1}$ and has cost $\sqrt{2} \cdot \frac{\ell_k(d2^k-1)}{\sqrt{2}}$ by Lemma \ref{spl}. However, we do not include the entirety of each edge to the split of terminal $\mathbf{p}$ in $T_k$; we only continue along the edge to $\mathbf{p}$ to the Fermat point of the triangle with the two new terminals corresponding to $\mathbf{p}$. As argued in Corollary \ref{cor: twice rep}, this removes a length of $\frac{1}{\sqrt{6}}$ per terminal in $T_k$. Finally, for each terminal in $T_{k+1}$ we connect it to the Fermat point of its respective triangle. Each such edge is of length $\sqrt{\frac{2}{3}}$ as argued in Corollary \ref{cor: twice rep}. Combining these quantities yields the numerator.

    If we assume that there exists $\ell = \lim_{k\to\infty}\ell_k$, then by taking limits of both sides of (\ref{eq:recursive}), we get
    
    \begin{equation*}
        \ell = \frac{\ell}{2\sqrt{2}}-\frac{1}{4\sqrt{3}}+\frac{1}{\sqrt{3}}.
    \end{equation*}
    
    Therefore, by expressing $\ell$, we obtain $$\ell=\frac{\sqrt{3}}{\sqrt{2}\left(2\sqrt{2}-1\right)}.$$
\end{proof}

To show that the limit exists, it is enough to show that $\ell_0 > \frac{\sqrt{3}}{\sqrt{2}(2\sqrt{2}-1)}$---it then follows from the recursive formula that $\ell_k >  \frac{\sqrt{3}}{\sqrt{2}(2\sqrt{2}-1)}$ and that the sequence $\{\ell_k\}_{k=0}^\infty$ is strictly decreasing. This holds for $T_0$ being the numerically computed optimal Steiner tree of the regular $d$-simplex for all $3 \leq d \leq 12$, for example.

\section{Progress towards Conjecture \ref{conj: simplex is the best graph embed}} \label{sec: simplex is the best graph embed}
 In this section, we consider Conjecture \ref{conj: simplex is the best graph embed}. For ease of notation throughout, we use $f: G(V,E) \to \R^{|V|}$ to denote the embedding of each edge in $G$ as its characteristic vector (e.g., $(i,j)$ is mapped to $\mathbf{e_i} + \mathbf{e_j}$). As evidence of the efficiency of the Steiner trees of regular simplices, we observe the following lemma.

\begin{lemma} \label{lem: diam 2}
For any fixed $m \geq 1$, the graph of size $m$ whose embedding (as above) has the minimum cost Steiner tree has diameter at most $2$.
\end{lemma}

That is, the only embeddings of graphs that might have more efficient Steiner trees than the embedding of the star graph (which embeds as a regular simplex) have diameter at most $2$. We actually prove an even stronger result. 
\begin{lemma} \label{lem: disjoint neighborhoods}
Let $G = (V,E)$ be a graph with $|E| = m$ with two vertices with disjoint closed neighborhoods. Then, there exists some $G'(V', E')$ with $|E'| =m$ with all closed neighborhoods of vertices pairwise overlapping such that $f(G')$ has a Steiner tree of total length less than the total length of the optimal Steiner tree of $f(G)$.
\end{lemma}
\begin{proof}
Without loss of generality, let $V  = [n]$. Let $i$ and $j$ be two vertices with disjoint neighborhoods. This implies that the two sets of endpoints of edges incident to $i$ and $j$ are disjoint. 

Now let $T$ be an optimum Steiner tree for $f(G)$ ($f$ is the embedding function described at the beginning of this section and in Section \ref{sec: intro}). We transform $T$ into a lower total length Steiner tree on $f(G')$ for $G'$ a graph with the neighborhoods of any pair of vertices overlapping.

For each point $\mathbf{x} = (x_1, x_2, \ldots x_n)$ in $T$, set the $i$\textsuperscript{th} coordinate equal to $\max(x_i, x_j)$ and then set the $j$\textsuperscript{th} coordinate equal to $0$. First, note that this operation maps each embedding of an edge incident to $j$, $f(k,j)$, to a distinct embedding of an edge incident to $i$ (namely, the embedding of the edge $(k,i)$). The resultant collection of embedded edges is the result of embedding $G$ after contracting the vertices $i$ and $j$ (call the contraction of $i$ and $j$ the graph $G'$): there are no lost or repeated edges exactly because $i$ and $j$ have disjoint closed neighborhoods. Second, this fixes all the other embedded edges in the configuration.

Now, consider two points $\mathbf{x} = (x_1, x_2, \ldots x_n)$ and $\mathbf{y} = (y_1, y_2, \ldots y_n)$ in $T$ with an edge between them. We want to show that the distance between them has not increased as a result of this map. The difference in each coordinate other than the $i$\textsuperscript{th} and $j$\textsuperscript{th} coordinates is fixed. So, it suffices to show that 
\[
    (x_i - y_i)^2 + (x_j - y_j)^2 \geq (\max(x_i, x_j) - \max(y_i, y_j))^2.
\]
Expanding both sides, we have to show that
\[
  x_i^2 + x_j^2 + y_i^2 + y_j^2   - 2x_iy_i - 2x_jy_j 
  \geq \max(x_i, x_j)^2 + \max(y_i, y_j)^2 - 2\max(x_i, x_j)\max(y_i, y_j).
\]
We have two cases to consider. First, suppose $\max(x_i, x_j) = x_i$ and $\max(y_i, y_j) = y_i$. Then we have
\begin{align*}
  (x_i - y_i)^2 + (x_j - y_j)^2
  &= (\max(x_i, x_j) - \max(y_i, y_j))^2 + (x_j - y_j)^2 \\
  &\geq (\max(x_i, x_j) - \max(y_i, y_j))^2.
\end{align*}
Notably, we have an equality in the second line only if $x_j = y_j$. The case of the maxima in the $j$\textsuperscript{th} coordinates follows symmetrically (with equality only if $x_i = y_i$).

Now suppose $\max(x_i, x_j) = x_i$ and $\max(y_i, y_j) = y_j$. First note that we have
\[
    (x_i - x_j)y_j \geq (x_i - x_j)y_i
  \]  
since $x_i \geq x_j$ and $y_j \geq y_i \geq 0$ (since $\mathbf{x}$ and $\mathbf{y}$ are in the convex hull of $f(G)$ by Theorem \ref{convexhull}). This implies
\begin{equation} \label{eqtn: swapped indices greater than same}
  2x_iy_j + 2x_jy_i \geq 2x_iy_i + 2x_jy_j.  
\end{equation}
 Now, 
 \begin{align*}
    x_i^2 + x_j^2 + y_i^2 + y_j^2   - 2x_iy_i - 2x_jy_j 
  &= \max(x_i, x_j)^2 + \max(y_i, y_j)^2 + x_j^2 + y_i^2  - 2x_iy_i - 2x_jy_j \\
  &= \max(x_i, x_j)^2 + \max(y_i, y_j)^2 + (x_j - y_i)^2 + 2x_jy_i  - 2x_iy_i - 2x_jy_j.
 \end{align*}
 Then, (\ref{eqtn: swapped indices greater than same}) implies that 
 \[
  (x_i - y_i)^2 + (x_j - y_j)^2 \geq (\max(x_i, x_j) - \max(y_i, y_j))^2,
 \]
 with equality exactly when $x_i = y_i = x_j$  or $x_j = y_j = y_i$. 
 The case of $\max(x_i, x_j) = x_j$ and $\max(y_i, y_j) = y_i$ follows by symmetry.
 
Finally, note that equality holds in either case only when the smaller of the $i$\textsuperscript{th} and $j$\textsuperscript{th} coordinates of the two points are of equal magnitude. But, consider a Steiner point $\mathbf{s}$ adjacent to terminal node $\mathbf{p}$. Such an incidence must occur by Lemma \ref{lem: max pt prop}. Since $i$ and $j$ are non-adjacent in $G$, either $p_i$ or $p_j$ are $0$. But then in particular $\min(p_i,p_j) = 0$, so, in order to have equality in the above, $s_i$ or $s_j$ must equal $0$, contradicting optimality as a result of Lemma \ref{lem: strict coordinate bounds}.
\end{proof}

\section*{Acknowledgments}

This work was carried out while the authors Guillermo A. Gamboa Q., Josef Matějka, and Jakub Petr were participants in the 2023 DIMACS REU program at Rutgers University, supported by CoSP, a project funded by European Union’s Horizon 2020 research and innovation programme, grant agreement No. 823748.
Karthik C.\ S.\ is supported by the National Science Foundation under Grant CCF-2313372 and by the Simons  
 Foundation, Grant Number 825876, Awardee Thu D. Nguyen.

\bibliographystyle{alpha}

\bibliography{references}

\newcommand{\etalchar}[1]{$^{#1}$}
\begin{thebibliography}{TGK{\etalchar{+}}16}

\bibitem[Aro98]{Arora_1998}
Sanjeev Arora.
\newblock {Polynomial time approximation schemes for Euclidean traveling
  salesman and other geometric problems}.
\newblock {\em Journal of the ACM}, 45(5):753–782, Sep 1998.

\bibitem[BGTZ14]{brazil2014history}
Marcus Brazil, Ronald~L. Graham, Doreen~A. Thomas, and Martin Zachariasen.
\newblock On the history of the {Euclidean Steiner} tree problem.
\newblock {\em Archive for history of exact sciences}, 68(3):327--354, 2014.

\bibitem[BRK{\etalchar{+}}12]{backes2012integer}
Christina Backes, Alexander Rurainski, Gunnar~W. Klau, Oliver M{\"u}ller,
  Daniel St{\"o}ckel, Andreas Gerasch, Jan K{\"u}ntzer, Daniela Maisel, Nicole
  Ludwig, Matthias Hein, et~al.
\newblock An integer linear programming approach for finding deregulated
  subgraphs in regulatory networks.
\newblock {\em Nucleic acids research}, 40(6):e43--e43, 2012.

\bibitem[CD13]{cheng2013steiner}
Xiuzhen Cheng and Ding-Zhu Du.
\newblock {\em Steiner trees in industry}, volume~11.
\newblock Springer Science \& Business Media, 2013.

\bibitem[CG76]{gilchung}
F.~R.~K. Chung and E.~N. Gilbert.
\newblock Steiner trees for the regular simplex.
\newblock {\em Bulletin of the Institute of Mathematics Academia Sinica},
  4:312--325, 1976.

\bibitem[Cho01]{cho2001steiner}
Jun-Dong Cho.
\newblock Steiner tree problems in {VLSI} layout designs.
\newblock In {\em Steiner Trees in Industry}, pages 101--173. Springer, 2001.

\bibitem[CR41]{CourantRobbins41}
R.~Courant and H.~Robbins.
\newblock {\em What is Mathematics?}
\newblock Oxford University Press, 1941.

\bibitem[CZ13]{Cheng_Zhang_2013}
Ya-Hong Cheng and Xiao-Dong Zhang.
\newblock The {Wiener} and terminal {Wiener} indices of trees.
\newblock {\em MATCH Communications in Mathematical and Computer Chemistry},
  70(2):591–602, 2013.

\bibitem[DEG01]{dobrynin2001wiener}
Andrey~A Dobrynin, Roger Entringer, and Ivan Gutman.
\newblock Wiener index of trees: theory and applications.
\newblock {\em Acta Applicandae Mathematica}, 66:211--249, 2001.

\bibitem[DJS86]{Day_Johnson_Sankoff_1986}
William H.~E. Day, David~S. Johnson, and David Sankoff.
\newblock The computational complexity of inferring rooted phylogenies by
  parsimony.
\newblock {\em Mathematical Biosciences}, 81(1):33–42, 1986.

\bibitem[DS96]{Du_Smith_1996}
Ding-Zhu Du and Warren~D. Smith.
\newblock Disproofs of generalized {Gilbert–Pollak} conjecture on the
  {Steiner} ratio in three or more dimensions.
\newblock {\em Journal of Combinatorial Theory, Series A}, 74(1):115–130, Apr
  1996.

\bibitem[DZ12]{Deng_Zhang_2012}
Xiaotie Deng and Jie Zhang.
\newblock Equiseparability on terminal {Wiener} index.
\newblock {\em Applied Mathematics Letters}, 25(3):580–585, Mar 2012.

\bibitem[FGK24]{Fleischmann2023}
Henry Fleischmann, Surya~Teja Gavva, and {Karthik {C. S.}}
\newblock On approximability of {Steiner} tree in $\ell_p$-metrics.
\newblock In {\em Proceedings of the 2024 Annual ACM-SIAM Symposium on Discrete
  Algorithms (SODA)}. SIAM, 2024.

\bibitem[GFP09]{Gutman_Furtula_Petrovic_2009}
Ivan Gutman, Boris Furtula, and Miroslav Petrovi\'{c}.
\newblock Terminal {Wiener} index.
\newblock {\em Journal of Mathematical Chemistry}, 46(2):522–531, Aug 2009.

\bibitem[GGJ77]{Garey_Graham_Johnson_1977}
Michael~R. Garey, Ronald~L. Graham, and David~S. Johnson.
\newblock The complexity of computing {Steiner} minimal trees.
\newblock {\em SIAM Journal on Applied Mathematics}, 32(4):835–859, 1977.

\bibitem[GJ77]{Garey_Johnson_1977}
Michael~R. Garey and David~S. Johnson.
\newblock The rectilinear {Steiner} tree problem is $np$-complete.
\newblock {\em SIAM Journal on Applied Mathematics}, 32(4):826–834, 1977.

\bibitem[GP68]{gilpol}
E.~N. Gilbert and H.~O. Pollak.
\newblock Steiner minimal trees.
\newblock {\em SIAM Journal on Applied Mathematics}, 16(1):1--29, 1968.

\bibitem[Han66]{Hanan_1966}
Maurice Hanan.
\newblock On {Steiner’s} problem with rectilinear distance.
\newblock {\em SIAM Journal on Applied Mathematics}, 14(2):255–265, 1966.

\bibitem[HR92]{hwang1992steiner}
Frank~K. Hwang and Dana~S. Richards.
\newblock Steiner tree problems.
\newblock {\em Networks}, 22(1):55--89, 1992.

\bibitem[HW13]{Humphries_Wu_2013}
Peter~J. Humphries and Taoyang Wu.
\newblock On the neighborhoods of trees.
\newblock {\em IEEE/ACM Transactions on Computational Biology and
  Bioinformatics}, 10(3):721–728, May 2013.

\bibitem[IOSS02]{ideker2002discovering}
Trey Ideker, Owen Ozier, Benno Schwikowski, and Andrew~F Siegel.
\newblock Discovering regulatory and signalling circuits in molecular
  interaction networks.
\newblock {\em Bioinformatics}, 18(suppl\_1):S233--S240, 2002.

\bibitem[IT12]{Ivanov_Tuzhilin_2012}
A.~O. Ivanov and A.~A. Tuzhilin.
\newblock The {Steiner} ratio {Gilbert–Pollak} conjecture is still open:
  Clarification statement.
\newblock {\em Algorithmica}, 62(1–2):630–632, Feb 2012.

\bibitem[JK34]{Jarnik1934}
Vojt\v{e}ch Jarn\'{i}k and Milo\v{s} K\"{o}ssler.
\newblock O minim\'{a}ln\'{i}ch grafech, obsahuj\'{i}c\'{i}ch $n$ daných
  bod\r{u}.
\newblock {\em \v{C}asopis pro p\v{e}stov\'{a}n\'{i} matematiky a fysiky},
  063(8):223--235, 1934.

\bibitem[Len12]{lengauer2012combinatorial}
Thomas Lengauer.
\newblock {\em Combinatorial algorithms for integrated circuit layout}.
\newblock Springer Science \& Business Media, 2012.

\bibitem[Lju21]{ljubic2021solving}
Ivana Ljubi{\'c}.
\newblock Solving {Steiner} trees: Recent advances, challenges, and
  perspectives.
\newblock {\em Networks}, 77(2):177--204, 2021.

\bibitem[LSW17]{Lee_Schmidt_Wright_2017}
Euiwoong Lee, Melanie Schmidt, and John Wright.
\newblock Improved and simplified inapproximability for k-means.
\newblock {\em Information Processing Letters}, 120:40–43, 2017.

\bibitem[Mit99]{M99}
Joseph S.~B. Mitchell.
\newblock Guillotine subdivisions approximate polygonal subdivisions: A simple
  polynomial-time approximation scheme for geometric tsp, k-mst, and related
  problems.
\newblock {\em SIAM Journal on Computing}, 28(4):1298--1309, 1999.

\bibitem[NR22a]{nayaki2022diagrammatic}
M~Puruchothama Nayaki and F.~Simon Raj.
\newblock Diagrammatic representation between topological indices and alkanes.
\newblock {\em Journal of Algebraic Statistics}, 13(2):3536--3545, 2022.

\bibitem[NR22b]{nayaki2022physical}
M~Puruchothama Nayaki and F.~Simon Raj.
\newblock The physical-chemical characteristics of alkanes and {Wiener}
  indices.
\newblock {\em Journal of Algebraic Statistics}, 13(2):3338--3345, 2022.

\bibitem[NRRK17]{noormohammadpour2017dccast}
Mohammad Noormohammadpour, Cauligi~S. Raghavendra, Sriram Rao, and Srikanth
  Kandula.
\newblock {DCCast}: Efficient point to multipoint transfers across datacenters.
\newblock In {\em 9th USENIX Workshop on Hot Topics in Cloud Computing
  (HotCloud 17)}, 2017.

\bibitem[RBK21]{Ramane_Bhajantri_Kitturmath_2021}
Harishchandra Ramane, Kavita Bhajantri, and Deepa Kitturmath.
\newblock Terminal status of vertices and terminal status connectivity indices
  of graphs with its applications to properties of cycloalkanes.
\newblock {\em Communications in Combinatorics and Optimization}, (Online
  First), Aug 2021.

\bibitem[RN10]{russakovsky2010steiner}
Olga Russakovsky and Andrew~Y. Ng.
\newblock A {Steiner} tree approach to efficient object detection.
\newblock In {\em 2010 IEEE Computer Society Conference on Computer Vision and
  Pattern Recognition}, pages 1070--1077. IEEE, 2010.

\bibitem[RP08]{resende2008handbook}
Mauricio G.~C. Resende and Panos~M. Pardalos.
\newblock {\em Handbook of optimization in telecommunications}.
\newblock Springer Science \& Business Media, 2008.

\bibitem[Smi92]{smith}
W.~D. Smith.
\newblock How to find {Steiner} minimal trees in {Euclidean} $d$-space.
\newblock {\em Algorithmica}, 7:137–177, 1992.

\bibitem[Spa96]{Spain96}
P.~G. Spain.
\newblock The {Fermat} point of a triangle.
\newblock {\em Mathematics Magazine}, 69(2):131--133, 1996.

\bibitem[Sul22]{Sulphikar_2022}
A~Sulphikar.
\newblock Computation of terminal {Wiener} index from subtrees.
\newblock {\em International Journal of Information Technology},
  14(6):3175–3181, Oct 2022.

\bibitem[SWW11]{Szekely2011}
L.A. Székely, Hua Wang, and Taoyang Wu.
\newblock The sum of the distances between the leaves of a tree and the
  ‘semi-regular’ property.
\newblock {\em Discrete Mathematics}, 311(13):1197--1203, 2011.
\newblock Selected Papers from the 22nd British Combinatorial Conference.

\bibitem[TGK{\etalchar{+}}16]{tuncbag2016network}
Nurcan Tuncbag, Sara J.~C. Gosline, Amanda Kedaigle, Anthony~R. Soltis, Anthony
  Gitter, and Ernest Fraenkel.
\newblock Network-based interpretation of diverse high-throughput datasets
  through the omics integrator software package.
\newblock {\em PLoS computational biology}, 12(4):e1004879, 2016.

\bibitem[Tre00]{Trevisan00}
Luca Trevisan.
\newblock When {Hamming} meets {Euclid}: The approximability of geometric {TSP}
  and {Steiner} tree.
\newblock {\em {SIAM} J. Comput.}, 30(2):475--485, 2000.

\bibitem[Wie47]{Wiener_1947}
Harry Wiener.
\newblock Structural determination of paraffin boiling points.
\newblock {\em Journal of the American Chemical Society}, 69(1):17–20, Jan
  1947.

\bibitem[ZEME16]{Zeryouh_El_Marraki_Essalih_2016}
Meryam Zeryouh, Mohamed El~Marraki, and Mohamed Essalih.
\newblock On the terminal {Wiener} index of networks.
\newblock In {\em 2016 5th International Conference on Multimedia Computing and
  Systems (ICMCS)}, page 533–536, Marrakech, Morocco, Sep 2016. IEEE.

\bibitem[ZME14]{Zeryouh_Marraki_Essalih_2014}
Meryam Zeryouh, Mohamed~El Marraki, and Mohamed Essalih.
\newblock {Wiener} and terminal {Wiener} indices of some rooted trees.
\newblock {\em Applied Mathematical Sciences}, 8:4995–5002, 2014.

\end{thebibliography}

\newpage

\appendix

\section{Inapproximability of the Euclidean Steiner Tree problem} \label{sec: simplicial complex conj}
In this section we formalize the reduction strategy for showing \apx-hardness of the Euclidean Steiner tree sketched in Section~\ref{sec: intro}.

Informally, we conjecture that regular simplicial complexes admit more efficient Steiner trees when they are composed of fewer simplices. Formally, we conjecture the following.

\begin{conjecture}[Euclidean Steiner Tree for Regular Simplicial Complexes] \label{con: simplicial complex Steiner trees}
For all constants $r \in (0,1)$ and $ \alpha \in (0, 1/r - 1)$, there exist constants $s, \beta > 0$ and $M \in \mathbb{Z}^+$ sufficiently large so that, for all $m \geq M$, given a regular, unit, simplicial complex on $m$ vertices:
\begin{enumerate}
    \item \textbf{Completeness:} If the vertices can be partitioned into the vertices of at most $rm$ unit, regular simplices, then the point configuration of the $m$ vertices admits an Euclidean Steiner tree of cost at most $sm$. 
    \item \textbf{Soundness:} If the vertices cannot be partitioned into the vertices of fewer than $(1 + \alpha)rm$ unit, regular simplices, then the point configuration of the $m$ vertices does not admit an Euclidean Steiner tree of cost less than $(1 + \beta)sm$.
\end{enumerate}
\end{conjecture}

In the above, for simplicity, we consider a single point to be a regular, unit simplex with one vertex. Conjecture \ref{con: simplicial complex Steiner trees} is entirely analytical; it does not directly involve any computation. Nonetheless, we show that if Conjecture \ref{con: simplicial complex Steiner trees} holds (or indeed a somewhat weaker conjecture holds), then the Euclidean Steiner tree problem is \apx-hard.
\begin{theorem} \label{thm: sufficient for apx hardness of Euclidean CST}
Conjecture \ref{con: simplicial complex Steiner trees} implies that the Euclidean Steiner tree problem is \apx-hard.
\end{theorem}
\begin{proof}
    We reduce from the Vertex Cover problem on triangle-free graphs. From \cite{Lee_Schmidt_Wright_2017}, there exists $r \in (0,1)$ and $ \alpha \in (0, 1/r - 1)$, and a family of $m$-edge, $n$-node graphs such that the following decision problem is \np-hard (where $n$ is a fixed function of $m$). 
Given an input graph $G$, decide which of the following cases holds.
\begin{itemize}
    \item \textbf{Completeness:} There exists a vertex cover of $G$ of size $rm$. 
    \item \textbf{Soundness:} All vertex covers of $G$ are of size at least $(1 + \alpha)rm$.
\end{itemize}

We will now describe a reduction from the Vertex Cover problem on triangle-free graphs to the Euclidean Steiner tree problem. Define $f_n : [n]^2 \to \R^n$ where $f_n(i,j) = \frac{1}{\sqrt{2}} \cdot (\mathbf{e_i} + \mathbf{e_j})$, the sum of the $i$\textsuperscript{th} and $j$\textsuperscript{th} standard basis vectors. Since the choice of domain will always be clear from context, we will abuse notation and denote $f_n$ by $f$. For a graph $G$ of order $n$, let 
\[
f(G) = f(E(G)) = \{f(i,j) \,:\, (i,j) \in E(G)\}.
\]

Now, given an input graph $G$ to the Vertex Cover problem on triangle-free graphs described above, our corresponding instance of the Euclidean Steiner Tree problem will be the instance on the terminal set $f(G)$. This mapping takes $O(\poly(m))$ time. Observe that $f(G)$ is exactly the collection of vertices of a regular, unit simplicial complex on $m$ vertices. Let $s$ and $\beta$ be as in Conjecture \ref{con: simplicial complex Steiner trees} for $r$ and $\alpha$ as in Conjecture \ref{con: simplicial complex Steiner trees} and $m$ sufficiently large.

\paragraph{Completeness.}
If the completeness case holds, i.e., $G$ admits a vertex cover $C$ of size $rm$, then the points in $f(G)$ can be partitioned into the vertices of at most $rm$ regular, unit simplices. Namely, if vertex $i \in C$, then the collection of points $S_i = \{f(i, j) \,:\, (i,j) \in E\} \subseteq f(G)$ (the embeddings of each edge incident to vertex $i$) forms the vertices of a regular, unit simplex. Since $C$ is a vertex cover of $G$, every edge in $G$ is incident to some vertex in $C$. Namely, $f(G) \subset \cup_{i \in C} S_i$. Since any subset of the vertices of a regular, unit simplex also forms the vertices of a regular, unit simplex (using the convention that a single vertex is the vertex of a regular, unit simplex on one vertex), any arbitrary partition of $f(G)$ among the $S_i$'s is a partitioning of $f(G)$ into the vertices of at most $rm$ unit, regular simplices. Hence, by the completeness case of Conjecture \ref{con: simplicial complex Steiner trees}, $f(G)$ admits an Euclidean Steiner tree of cost at most $sm$.

\paragraph{Soundness.} Now suppose that the soundness case holds, i.e., all vertex covers of $G$ are of size at least $(1 + \alpha)m$. We need to show that the points in $f(G)$ cannot be partitioned into the vertices of fewer than $(1 + \alpha)rm$ regular, unit simplices (and, hence, the soundness case of Conjecture \ref{con: simplicial complex Steiner trees} applies). To do this, we make a series of claims.

\begin{claim} \label{cla: unit claim 1}
For $\{i,j\}, \{k,\ell\} \in E(G)$ such that $\{i,j\} \cap \{k, \ell\} = \emptyset$, $f(i,j)$ and $f(k, \ell)$ cannot belong to the same regular, unit simplex in any partition of $f(G)$ into the vertices of regular, unit simplices.
\end{claim}
\begin{proof}
Note that $\norm{f(i,j) - f(k, \ell)}_2 = \sqrt{2} \neq 1$. 
\end{proof}
\begin{claim} \label{cla: unit claim 2}
For $S \subset E(G)$ with $|S| \neq \emptyset$ such that $\cap_{e \in S} e = \emptyset$, the points in $f(S)$ cannot all belong to the same regular, unit simplex in any partition of $f(G)$ into the vertices of regular, unit simplices.
\end{claim}
\begin{proof}
The case of $|S| = 1$ is trivial and the case of $|S| = 2$ follows immediately from Claim \ref{cla: unit claim 1}.
 
Now assume that $|S| \geq 3$. Suppose that $e_1 = \{i,j\} \in S$. By Claim \ref{cla: unit claim 1}, for all $e \in S$, either $i \in e$ or $j \in e$. Since $\cap_{e \in S} e = \emptyset$, there exists $e_2 \in S$ such that $i \in e_2$ and $j \not \in e_2$  and $e_3 \in S$ such that $i \not\in e_3$ and $j \in e_3$. Now, by Claim \ref{cla: unit claim 1} again, $e_2$ and $e_3$ must share a vertex, so $e_2 = \{i,k\}$ and $e_3 = \{j,k\}$. But, $G$ is triangle-free and $e_1, e_2, $ and $e_3$ form a triangle, yielding a contradiction.
\end{proof}

Now, Claim \ref{cla: unit claim 2} implies that in any partition $E_1 \sqcup E_2 \cdots $ of $E(G)$ corresponding to a partition of $f(G)$ into the vertices of regular, unit simplices, for each part $E_i$, there exists $v_i \in V$ such that $v_i \in \cap_{e \in E_i} e$. Indeed, the $v_i$'s form a vertex cover of $G$, implying that $G$ has a vertex cover of size at most the size of the partition of $f(G)$ into the vertices of regular, unit simplices. Hence, by our assumption in the soundness case of our hard instance of the Vertex Cover problem, the vertices in $f(G)$ cannot be partitioned into the vertices of fewer than $(1 + \alpha) rm$ unit, regular simplices. Then, by the soundness case of Conjecture \ref{con: simplicial complex Steiner trees}, $f(G)$ does not admit an Euclidean Steiner Tree of cost less than $(1 + \beta)sm$ in this case. \\

Combining the analysis of the completeness and soundness cases, the Euclidean Steiner Tree problem is \np-hard to approximate within a factor of less than $(1 + \beta)$, yielding the desired result.
\end{proof}
Note that we only used a weaker version of Conjecture \ref{con: simplicial complex Steiner trees} to prove Theorem \ref{thm: sufficient for apx hardness of Euclidean CST}. Indeed, we really only need that $s$ and $\beta$ exist for $r$ and $\alpha$ induced by the inapproximability of Vertex Cover on triangle-free graphs.
\end{document}